\documentclass[a4paper,leqno,11pt]{amsart}

\usepackage{epsf,graphicx,epsfig}
\usepackage{amscd}
\usepackage{amsmath,latexsym,amssymb,amsthm}
\usepackage[nospace,noadjust]{cite}
\usepackage{setspace,cite}
\usepackage{lscape,fancyhdr,fancybox}
\usepackage{graphicx}

\usepackage{color}
\usepackage{url}
\usepackage{enumerate}
\usepackage[mathscr]{euscript}
  \theoremstyle{plain}

\swapnumbers
    \newtheorem{thm}{Theorem}[section]
    \newtheorem{prop}[thm]{Proposition}
   \newtheorem{lemma}[thm]{Lemma}
    \newtheorem{corollary}[thm]{Corollary}
    
    \newtheorem{subsec}[thm]{}
\theoremstyle{definition}
    \newtheorem{defn}[thm]{Definition}
    \newtheorem{exam}[thm]{Example}

\theoremstyle{remark}
     \newtheorem{remark}[thm]{Remark}

\title{}
\author{}
\date{}
\usepackage{amssymb}

\begin{document}
\title{Modular Class of a Lie algebroid with a Nambu Structure}
\author{Apurba Das}
\email{ apurbadas348@gmail.com}
\address{Stat-Math Unit,
 Indian Statistical Institute, Kolkata 700108,
West Bengal, India.}

\author{Shilpa Gondhali}
\email{shilpa.s.gondhali@gmail.com}
\address{Department of Mathematics,
 University of Haifa, Israel.}

\author{Goutam Mukherjee}
\email{gmukherjee.isi@gmail.com}
\address{Stat-Math Unit,
 Indian Statistical Institute, Kolkata 700108,
West Bengal, India.}

\date{\today}
% \classification{Primary: 55N22, 55N91, ;\ Secondary: 55P91, 55Q91,55M35}
%\keywords{}

\thispagestyle{empty}

\begin{abstract}
In this paper, we introduce the notion of modular class of a Lie algebroid $A$ equipped with a  Nambu structure satisfying some suitable hypothesis. We also introduce cohomology and homology theories for such Lie algebroids and prove that these theories are connected by a duality isomorphism when the modular class is null.

\end{abstract}
\maketitle
\section{Introduction}
For a Poisson manifold $P$, it is well known that the cotangent bundle $T^\ast P$ has a natural Lie algebroid structure (see \cite{vais} for details). It is well known that for a Poisson manifold $P$ there is a distinguished characteristic class, called the modular class of $P$ which is defined as a first Lie algebroid cohomology class of the Lie algebroid $T^\ast P$. Nambu-Poisson manifolds are generalization of Poisson manifolds and just like Poisson manifolds, the notion of modular class of a Nambu-Poisson manifold exists. In \cite {ilmp}, the authors proved that for an oriented Nambu-Poisson manifold $M$ of order $n,~ n\geq 3$, the vector bundle $\Lambda ^{n-1}(T^*M)$ is a Leibniz algebroid and as a result the space $\Omega^{n-1}(M)$ of smooth sections of the bundle  $\Lambda ^{n-1}(T^*M)$ is a Leibniz algebra (a non anti-symmetric version of Lie algebra). The modular class of $M$ is then defined as an element of the first Leibniz algebroid cohomology of $\Lambda ^{n-1}(T^*M)$ which is by definition, the first cohomology class of the Leibniz algebra  $\Omega^{n-1}(M)$ with coefficients in $C^{\infty}(M).$ In \cite{illmp}, the authors introduced homology and cohomology theories for Nambu-Poisson manifolds and when the modular class vanishes, the authors proved a duality theorem connecting  Nambu cohomology and canonical Nambu homology modules. 

The notion of a Nambu structure on a Lie algebroid was introduced in \cite {wade} as a generalization of the notion of a Nambu-Poisson manifold. It turns out that for a Lie algebroid $A$ over a smooth manifold $M$ equipped with a Nambu structure of order $n$, the vector bundle $\Lambda^{n-1}A^\ast $ is a Leibniz algebroid. The aim of this article is to generalize the methods developed in \cite {illmp}, \cite{ilmp}, in the context of Lie algebroid equipped with a Nambu structure. To do this, we introduce modular class  of a Lie algebroid $A$ equipped with a Nambu structure of order $n$ as an element of the first Leibniz algebroid cohomology of the associated Leibniz algebroid $\Lambda^{n-1} A^\ast$, under the assumption that the space of smooth sections of $A^\ast$ is locally spanned by elements of the form $d_Af,~~ f\in C^\infty(M)$, where $d_A$ denotes the coboundary operator in the Lie algebroid cohomology complex of $A$ with trivial coefficients.  Moreover, we introduce Nambu cohomology and  Nambu homology theories, generalizing the corresponding results for Nambu-Poisson manifolds, in the context of Lie algebroids equipped with regular Nambu structures and prove that if the modular class is null then these theories are connected by  duality isomorphisms. It may be mentioned that in \cite {elw}, the authors introduced a notion of modular class of an arbitrary Lie algebroid which arise from a representation on a specific line bundle, extending the notion of the modular class of Poisson manifolds.

The paper is organized as follows. In $\S 2$, we review some known facts and fix our notation. In $\S3$, we define the notion of modular class of a Lie algebroid which is orientable as a vector bundle and equipped with a Nambu structure. In $\S4$, we first introduce Nambu cohomology for a Lie algebroid $A$ equipped with a regular Nambu structure and give an equivalent formulation of this cohomology which is a variant of the notion of foliated cohomology \cite{vais71, vais73}. Next, under the assumption that $A$ is oriented as a vector bundle, we introduce the notion of canonical Nambu homology theory of $A$. Finally, in this section we connect these theories by proving a duality isomorphism theorem under the assumption that the modular class of $A$ is null. In $\S5$, the last section of the paper, we show that we can do away with the orientability assumption on the vector bundle $A$ and introduce modular class and recover the results of $\S 4$ by using the notion of density on $A$.
\section{Preliminaries}\label{prel}
In this section we briefly recall some known facts and fix notations (see \cite{mac} for details). All vector bundles we consider are smooth vector bundles over smooth paracompact manifolds. The notion of Lie algebroids was introduced by J. Pradines \cite{pr} as generalization of tangent bundles of smooth manifolds and Lie algebras.

\begin{defn} \label{lie-d} 
Let $M$ be a smooth manifold. A {\it{Lie algebroid}} $(A,[,]_A,a)$ over $M$ is a vector bundle $A$ over $M$ together with a
vector bundle map $a \colon A \longrightarrow TM$, called {\it{anchor}} of $A$ and a
bilinear map $[,]_A\colon \Gamma A \times \Gamma A \longrightarrow \Gamma A$ on the space  $\Gamma A$ of smooth sections of $A$, which makes $\Gamma A$ a Lie algebra such
that for all $X, Y \in \Gamma A$ and $f\in C^\infty(M)$ following holds:
$$[X,fY]_A= f[X,Y]_A+ a(X)(f)Y.$$
\end{defn}

Recall that the $ C^\infty(M)$-linear map $\Gamma A \longrightarrow \chi (M)$ induced by $a$, where $\chi (M)$ denote the space of vector fields on $M,$ is a Lie algebra homomorphism, that is,
$$a([X,Y]_A)=a(X)a(Y)-a(Y)a(X)=[a(X),a(Y)].$$

A representation $\rho \colon A \longrightarrow \mathcal{D}(E)$ of a Lie algebroid $A$ over $M$ on a vector bundle $E$ over $M$ is a Lie algebroid morphism, where $\mathcal{D}(E)$ is the Lie algebroid of derivations on $E$. In particular, the representation  of $A$ on the trivial bundle $M\times\mathbb{R}\longrightarrow M$
given by $$a^0: (X,f) \mapsto  a(X)(f)$$ for all $f\in C^\infty(M)= \Gamma(M\times \mathbb{R})$ and
$X\in\Gamma A$ is called the {\it{trivial representation}}.

We now briefly recall the definition of cohomology of a Lie algebroid with coefficients in a given representation. Given a representation $\rho$ on a vector bundle $E$ over $M$ of a Lie algebroid $A$ over $M$, we have a sequence of $C^\infty(M)$- modules $\Gamma (\Lambda^kA^\ast\otimes E), ~~k\geq 0,$ and $\mathbb R$-linear maps
$$d_A \colon \Gamma (\Lambda^kA^\ast\otimes E) \longrightarrow \Gamma (\Lambda^{k+1}A^\ast\otimes E)$$ which forms a cochain complex. The homology of this complex defines the Lie algebroid cohomology of $A$ with coefficients in $\rho$. Note that an element of $\Gamma (\Lambda^kA^\ast\otimes E)$ may be viewed as a $C^{\infty}(M)$-multilinear and alternating map
$$\phi\colon \underbrace{\Gamma A\times\cdots\times \Gamma A}_{k\hspace*{1mm}\text{times}} \longrightarrow \Gamma E$$ and then $d_A\phi$ is given by the formula
\begin{multline*}
 d_A\phi(X_1,\cdots, X_{k+1})=
\displaystyle{\sum_{r=1}^{n+1}}(-1)^{r+1}\rho(X_r)(\phi(X_1,\cdots,X_{r-1},\widehat{X_r},X_{r+1},\cdots,X_{k+1}))\\+
\displaystyle{\sum_{1\leq r<s\leq n+1}}(-1)^{r+s}\phi([X_r,X_s]_A,X_1,\cdots,\widehat{X_r},\cdots,\widehat{X_s},\cdots,X_{k+1}).
\end{multline*}
\noindent for all $X_1,\cdots,X_{k+1}\in \Gamma A$. 
The $k^{th}$-cohomology of $A$ with coefficients in $\rho$ is denoted by $\mathcal{H}^k(A,\rho)$. In particular, when $E= M\times \mathbb R$ and $\rho = a^0$, we denote the above cochain complex simply by $\{\Gamma (\Lambda^k A^\ast), d_A\}_{k\geq 0}$ and the corresponding cohomology is denoted by $\mathcal{H}^\bullet(A).$

Recall that we have two useful operators defined as follows \cite{mac}

\begin{defn}\label{lie-L-i}
Let $M$ be a smooth manifold. Let $(A,[,]_A,a)$ be a Lie algebroid over $M$.\\
(1) Let $X\in \Gamma A$. The {\it{Lie derivative}}
$\mathcal{L}_X\colon \Gamma(\Lambda^nA^\ast)\longrightarrow  \Gamma(\Lambda^nA^\ast)$ is defined by
\[
 \mathcal{L}_X(\phi)(X_1,\cdots, X_n)= a(X)(\phi(X_1,\cdots,X_n))
-\displaystyle{\sum_{r=1}^n}\phi(X_1,\cdots,[X,X_r]_A,\cdots,X_n)
\]
\noindent where $\phi\in \Gamma(\Lambda^n A^\ast), X_r\in \Gamma A$ for $r=1,\cdots,n$.\\
\noindent (2) Let $X\in \Gamma A$.
The {\it{interior multiplication }}, also known as {\it{contraction}}
$$\iota_X\colon  \Gamma(\Lambda^{n+1} A^\ast)\longrightarrow \Gamma(\Lambda^n A^\ast)$$ is defined by
\[
 \iota_X(\phi)(X_1,\cdots,X_n)= \phi(X,X_1,\cdots,X_n)
\]
\noindent where $n\in\mathbb{N}, \phi\in  \Gamma(\Lambda^{n+1} A^\ast)$ and $X_1,\cdots,X_n\in \Gamma A$.\\
\end{defn}

The operators $\mathcal{L}_X$, $\iota_X$ and $d_A$ satisfy a set of formulas similar to those which hold in the calculus of vector-valued forms on manifolds (see \cite{mac} for details). 

Let `$\wedge$' be the standard wedge product of $k$-multisections. For $\phi \in \Gamma (\Lambda^kA^\ast)$ and  a $k$-multisection $\xi \in \Gamma (\Lambda^kA)$, let $\langle \phi, \xi\rangle$ denote the standard pairing. Then the contraction operator extends to give a well defined map $\iota_{\xi} : \Gamma(\Lambda^kA^\ast) \longrightarrow \Gamma (\Lambda^{k-i}A^\ast)$ for $\xi\in \Gamma(\Lambda^i A)$ verifying
\[
 \langle \iota_\xi(\phi),\eta \rangle= (-1)^{[{\frac{i}{2}}]}\langle \phi, \xi\wedge \eta\rangle
\]
for $\xi \in \Gamma (\Lambda^iA),~~ \eta \in \Gamma (\Lambda^j A),~~ \phi \in \Gamma (\Lambda^{i+j} A^\ast) .$

\noindent

Next, given $\phi\in \Gamma A^\ast,$
define $\iota_\phi\colon \Gamma (\Lambda^kA) \longrightarrow \Gamma (\Lambda^{k-1}A)$ for $\xi\in\Gamma (\Lambda^k A)$ and $\psi\in\Gamma (\Lambda^{k-1}A^\ast)$ by
\[
 \langle \psi, \iota_\phi(\xi)\rangle= \langle \phi\wedge\psi, \xi\rangle .
\]
This again extends to a well-defined map $\iota_\phi\colon\Gamma (\Lambda^kA) \longrightarrow\Gamma (\Lambda^{k-i}A)$
for $\phi\in\Gamma (\Lambda^iA^\ast)$. 

For $ \phi \in \Gamma (\Lambda^iA^\ast), ~~\psi\in\Gamma (\Lambda^jA^\ast),~~  \xi\in\Gamma (\Lambda^{i+j}A),$ 
we have
\[
 \langle\psi,\iota_\phi(\xi)\rangle= (-1)^{[\frac{i}{2}]}\langle\phi\wedge\psi,\xi\rangle.
\]
In particular, when $\xi\in \Gamma (\Lambda^k A)$ is of the form $\xi = X_1 \wedge \cdots X_k$ then $\iota_{\xi}$ is simply the compositions of the operators $\iota_{X_i},~~1\leq i \leq k.$ Similar is the case for $\iota _{\phi}$ for $\phi\in \Gamma (\Lambda^iA^\ast).$
Moreover, for $X\in \Gamma A,~~\theta \in \Gamma A^\ast,~~ \phi, ~ \psi \in \Gamma (\Lambda^\bullet A^\ast) $ and $\xi, ~\eta \in \Gamma (\Lambda^\bullet A)$ with $\phi$ and $\xi$ of degree $i$, the following hold:
$$\iota_X(\phi \wedge \psi) = \iota (\phi) \wedge \psi + (-1)^i \phi \wedge \iota_X(\psi),$$
$$\iota _{\theta}(\xi \wedge \eta ) = \iota (\xi ) \wedge \eta + (-1)^i\xi \wedge \iota_{\theta}(\eta).$$

The notion of  Leibniz algebras was introduced by J-L Loday \cite {loday1}, as a non anti-symmetric analogue of Lie algebras. Recall that a left Leibniz bracket on a real vector space $L$ is a $\mathbb R$-bilinear operation $\ll~,~\gg_L : L \times L \longrightarrow L$ such that the following identity, known as Leibniz identity
$$\ll a, \ll b, c\gg_L\gg_L = \ll \ll a, b\gg_L, c\gg_L + \ll b, \ll a, c\gg_L\gg_L$$ hold for all $a,~b,~c \in L$. The vector space $L$ equipped with a left Leibniz bracket is called a left Leibniz algebra. Through out the paper, by a Leibniz algebra, we shall always mean a left Leibniz algebra. Note that the bracket $\ll ~,~\gg_L$ becomes a Lie bracket, if in addition, it is anti-symmetric. Thus every Lie algebra is, in particular, a Leibniz algebra.

Recall \cite{loday2} that a {\it representation of a Leibniz algebra} $L$ is a $\mathbb R$-module $\mathcal E$ equipped with a $\mathbb R $-bilinear map
$$ L \times \mathcal E \longrightarrow \mathcal E,~~ (a, e) \mapsto ae,$$ such that $a_1(a_2e) = \ll a_1, a_2\gg_Le + a_2(a_1e),$ for all $a_1,~a_2 \in L$ and $e \in \mathcal E.$    

Leibniz algebroids are generalizations of Leibniz algebras. 
\begin{defn}\label{leibniz-d}
Let $M$ be a smooth manifold. A {\it{Leibniz algebroid}} over $M$ is a vector bundle $A$ over a smooth manifold $M$ together with a
vector bundle map $a \colon A \longrightarrow TM$ over $M$, called {\it{anchor}} of $A$ and a
$\mathbb{R}$-bilinear map $\ll,\gg_A\colon \Gamma A\times\Gamma A\longrightarrow\Gamma A$ such that
for all $X,Y,Z\in \Gamma A$ and $f\in C^\infty(M)$ following hold.
\begin{enumerate}
\item $\ll X, fY\gg_A=  f\ll X,Y\gg_A + a(X)(f)Y,$
\item $\ll X,\ll Y,Z\gg_A\gg_A= \ll \ll X,Y\gg_A,Z\gg_A+ \ll Y,\ll X,Z\gg_A\gg_A.$
\end{enumerate}
\end{defn}
\noindent
It is a consequence (cf. Proposition $3.1$, \cite{db}) of the Definition~\ref{leibniz-d} that 
$$a(\ll X,Y\gg_A)= [a(X),a(Y)]_{TM}, ~~\mbox{for all}~~X, ~Y \in \Gamma A.$$

We will often use the notation $(A, \ll ~,~\gg_A, a)$ to denote a Leibniz algebroid over a smooth manifold $M$.

%----------------------------------------------------------------------------------------------------------
\newpage
\begin{exam}\label{leibniz-e}
\noindent
\begin{enumerate}
\item Any Lie algebroid is a Leibniz algebroid.
\item Every  Leibniz algebra may be considered as a Leibniz algebroid over a one point space.
\item Let $M$ be a smooth manifold and set $A= TM\oplus T^\ast M$. Then $A$ is a vector bundle over $M$. Define $\ll ,\gg_A\colon \Gamma A\times\Gamma A\longrightarrow \Gamma A$ as
\[
 \ll X+\xi,Y+\eta\gg_A= [X,Y]_{TM}+ \mathcal{L}_X\eta- \iota_Y(d\xi)
\]
where $X,Y\in \Gamma (TM)= \mathfrak{X}(M)$ and $\xi,\eta\in\Gamma (T^\ast M)= \Omega^1(M)$. Then it is easy to check that $A$ is a Leibniz algebroid with the projection onto the first factor as the anchor.
\end{enumerate}
\end{exam}
The notion of morphism between two Leibniz algebroids over the same base manifold is similar to the definition of morphism between Lie algebroids.  
Just like Lie algebroids, the notion of representation of a Leibniz algebroid may be defined as follows. 
\begin{defn} \label{leibniz-r}
A {\it{representation}} of a Leibniz algebroid $(A, \ll ~,~ \gg_A, a)$ over $M$ on a vector bundle $E$  over $M$  is a Leibniz algebroid morphism 
$$\rho \colon A \longrightarrow  \mathcal{D}(E),$$ where $\mathcal{D}(E)$ is the Lie algebroid of derivations on $E$.
\end{defn}

In other words, a representation of $A$ consists of a $\mathbb{R}$-bilinear map
$$\Gamma A\times \Gamma E\longrightarrow \Gamma E;~~(X,s)\mapsto \rho_X s$$
\noindent
such that for any $X,Y\in \Gamma A,~ f\in C^\infty(M)$ and $s\in \Gamma E$ the following hold:
\begin{enumerate}
\item $\rho_{fX}s=f\rho_X s,$
\item $\rho_X(fs)= f\rho_Xs+ (a(X)f)s,$
\item $\rho_X(\rho_Ys)-\rho_Y(\rho_Xs)= \rho_{\ll X,Y\gg_A}s$.
\end{enumerate}
Note that when considered over a point, the above notion is precisely the notion of the representation of a Leibniz algebra.

%------------------------------------------------------------------------------------------------------------

\begin{exam}\label{leibniz-r-e}
\noindent
Let $A$ be a Leibniz algebroid on $M$. The representation $a^0$ of $A$ on the trivial line bundle
$$M\times\mathbb{R}\longrightarrow M$$
given by $a^0(X)(f)= a(X)(f)$ for all $f\in C^\infty(M)$ and $X\in\Gamma A$ is called the {\it{trivial
representation}}.
\end{exam}

%----------------------------------------------------------------------------------------------------------
\begin{defn} \label{leibniz-co}
Let $(A,\ll,\gg_A,a)$ be a Leibniz algebroid over $M$. Let $\rho$ be a representation of $A$ on a vector bundle $E$ over $M$.
Define a sequence of complex as follows.\\
For $n\in\mathbb{N}$,
\[
C^n(A,\rho):= \big\{\phi\colon\underbrace{\Gamma A\times\cdots\times \Gamma A}_{n\hspace*{1mm}\text{times}}\longrightarrow \Gamma E|~~ \phi\hspace*{1mm}\text{is}\hspace*{1mm}
{\mathbb R}\text{-multilinear}\big\}.
\]
\noindent Define coboundary operator
$d_{A,\rho}\colon C^n(A,\rho)\longrightarrow C^{n+1}(A,\rho)$ by the formula
\begin{multline*}
 d_{A,\rho} \phi(X_1,\cdots, X_{n+1})=
\displaystyle{\sum_{r=1}^{n+1}}(-1)^{r+1}\rho_{X_r}(\phi(X_1,\cdots,X_{r-1},\widehat{X_r},X_{r+1},\cdots,X_{n+1}))\\+
\displaystyle{\sum_{1\leq r<s\leq n+1}}(-1)^{r}\phi(X_1,\cdots,\widehat{X_r},\cdots,\ll X_r,X_s\gg_A,\cdots,X_{n+1}).
\end{multline*}
\noindent for all $\phi\in C^n(A;\rho)$ and $X_1,\cdots,X_{n+1}\in \Gamma A$. One can check that
$d_{A,\rho}\circ d_{A,\rho}= 0$.
\\ Let
\begin{gather*}
\mathcal{Z}^n(A,\rho):= \text{ker} (d_{A,\rho}\colon C^n(A,\rho)\longrightarrow C^{n+1}(A,\rho))\\
\mathcal{B}^n(A,\rho):= \text{im} (d_{A,\rho}\colon C^{n-1}(A,\rho)\longrightarrow C^n(A,\rho))
\end{gather*}
\noindent where $\mathcal{B}^0(A,\rho)= \{0\}$ and $n\in\mathbb{N}$. Then the $n^{th}$-cohomology of $A$ with coefficients in the representation $\rho$ is defined by the
quotient
\[
 \mathcal{H}L^n(A,\rho)= \mathcal{Z}^n(A,\rho)/\mathcal{B}^n(A,\rho).
\]
\end{defn}
\noindent {\bf Notation:} In the case, when the representation $\rho$ is the trivial representation, we denote the cohomology operator by $d_A$ and the cohomology modules simply by $\mathcal{H}L^\bullet(A).$
%------------------------------------------------------------------------------------------------------------------
%\newpage
\begin{remark}\label{leibniz-c-r}
\noindent
\begin{enumerate}
\item In the case, when $M$ is a point, then $A$ is a Leibniz algebra and $\rho$ reduces to a representation of $A$ as a Leibniz algebra and in this case, the definition of cohomology reduces to the Leibniz algebra cohomology with coefficients in a representation.
\item When $\rho$ is the trivial representation then the cohomology as defined above is precisely the Leibniz algebroid cohomology as introduced in \cite{ilmp}.
\end{enumerate}
\end{remark}

\section{lie algebroid with a nambu structure and the modular class}

\noindent Let $M$ be a smooth manifold of dimension $m$. Recall that a  {\it{Nambu-Poisson bracket}} of order $n$, $n\leq m$, on $M$ is an $n$-multilinear
mapping 
$$\{,\ldots,\}\colon \underbrace{C^\infty(M)\times \cdots\times C^\infty(M)}_{n\hspace*{1mm}\text{times}}
\longrightarrow C^\infty(M)$$ satisfying the following:
\begin{enumerate}
\item Alternating:
\[
 \{f_1,\ldots,f_n\}= (-1)^{\varepsilon(\sigma)}\{f_{\sigma(1)},\ldots,f_{\sigma(n)}\},
\]
\noindent for all $f_1,\ldots,f_n\in C^\infty(M)$ and $\sigma\in \Sigma_n$, where $\Sigma_n$
is the symmetric group of $n$ elements and $\varepsilon(\sigma)$ is the parity of the permutation $\sigma$,\\
\item Leibniz rule: 
$$\{f_1g_1,f_2,\ldots,f_n\}= f_1\{g_1,f_2,\ldots,f_n\}+ g_1\{f_1,f_2.\ldots,f_n\},$$
\item Fundamental identity: 
$$\{f_1,\ldots,f_{n-1},\{g_1,\ldots,g_n\}\}=
\displaystyle{\sum_{i=1}^n}\{g_1,\ldots,\{f_1,\ldots,f_{n-1},g_i\},\ldots,g_n\}$$
for all $f_1,\ldots,f_{n-1},g_1,\ldots,g_n\in C^\infty(M)$.
\end{enumerate}
\noindent
Given a Nambu-Poisson bracket, one can define an $n$-vector field $\Lambda\in\Gamma (\Lambda^n TM)$ as
\[
 \Lambda(df_1,\ldots,df_n)=\{f_1,\ldots,f_n\},
\]
\noindent for $f_1,\ldots,f_n\in C^\infty(M)$. The pair $(M,\Lambda)$ is called a
{\it{Nambu-Poisson manifold of order $n$}}. Nambu structure on a Lie algebroid (\cite{wade}) is a generalization of Nambu-Poisson manifold. In this section, we introduce the notion of the modular class of a Lie algebroid equipped with a Nambu structure under some suitable assumptions.

Recall that for a Lie algebroid $(A,[,]_{A},a)$ over a smooth manifold M, the algebra $\Gamma(\Lambda^\bullet A)$ endowed with the standard wedge product $\wedge$ and the generalized Schouten bracket extending the bracket on $\Gamma A$ (also denoted by $[~,~]_A$) is a Gerstenhaber algebra.
\begin{defn}
Let $M$ be a smooth manifold. Let $(A,[,]_{A},a)$ be a Lie algebroid over $M$. Let $n\in\mathbb{N},~n \leq m =~ rank~ A$ and
$\Pi\in \Gamma(\Lambda^n A)$ a smooth section of the vector bundle $\Lambda^n A.$ We say that $\Pi$ is a {\it{Nambu structure}} of order $n$ if
\[
 [\Pi\alpha, \Pi]_{A}\beta= -\Pi(\iota_{\Pi\beta}d_{A}\alpha), \hspace*{3mm}\text{for any}\hspace*{1mm}
\alpha,\beta\in \Gamma(\Lambda^{n-1}A^\ast).
\]
 \noindent where $d_{A}$ denote coboundary operator for the Lie algebroid cohomology of $A$ with trivial coefficients and $\Pi\alpha:=
\iota_\alpha\Pi$ for all $\alpha\in\Gamma (\Lambda^\bullet A^\ast)$.
\end{defn}
%-------------------------------------------------------------------------------------------------------------
\begin{remark}\label{associated_nambu_poisson}
The notion Nambu structure on Lie algebroids is a generalization of Nambu-Poisson manifolds in the sense that
given a Lie algebroid
$(A,[,]_A,a)$ over $M$ equipped with a Nambu structure, there is a natural Nambu-Poisson bracket on $C^\infty(M)$ arising from the given Nambu
structure and conversely, for any Nambu-Poisson manifold $(M,\Lambda)$ of order $n$, the $n$-vector field  $\Lambda\in\Gamma (\Lambda^n TM)$ satisfies the condition of the above definition with the Schouten-Nijenhuis bracket $[\cdot,~\cdot]_{SN}$ on multi-vector fields \cite{bv} and $d$ the de Rham differential operator and hence $\Lambda$ is a Nambu- structure on $TM$.
\end{remark}

\begin{defn}
\noindent

(i) Let $(A,[,]_{A},a)$ be a Lie algebroid over a smooth manifold $M$. A smooth section $\Pi\in  \Gamma(\Lambda^j A)$ is called {\it{locally decomposable}} if for any $x\in M$, either $\Pi(x)= 0$ or in a neighborhood of $x$, $\Pi$ can be expressed as $\Pi= X_1\wedge\cdots\wedge X_j$ where
$X_1,\ldots,X_j$ are local sections of $A$ defined on that neighbourhood.
\noindent

(ii) A point $x\in M$ is called a {\it{singular point}} of $\Pi$ if
$\Pi(x)=0$ and is called a {\it{regular point}} of $\Pi$ if $\Pi(x)\neq 0$.\\
\end{defn}

%-------------------------------------------------------------------------------------------------------------
\begin{defn}
 Let $(A,[,]_{A},a)$ be a Lie algebroid over $M$ with a Nambu structure $\Pi$ of order $n$.
We say $\Pi$ is a {\it{regular Nambu structure}} if $\Pi(x)\neq 0$ for all $x\in M$ i.e. each point of $M$
is a regular point of $\Pi$.
\end{defn}

%-----------------------------------------------------------------------------------------------------------
\noindent

\begin{remark}\label{linearly_independent}
Let $(A,[,]_{A},a)$ be a Lie algebroid over a smooth manifold $M$. Let $\Pi\in \Gamma (\Lambda^j A),$ where $j\in\mathbb{N}, j < \mbox{rank}~ A$. Let $x\in M$ be a regular point of $\Pi$. Let $\Pi$ be locally decomposable at $x$. Then we may choose a neighbourhood $U$ of $x$ and linearly independent local sections $X_1,\ldots,X_j$ of $A$ defined on $U$ such that $\Pi\mid_{U}$ is of the form $\Pi\mid_{U} = X_1\wedge\cdots\wedge X_j.$
\end{remark}
We have the following result \cite{wade}.
%-------------------------------------------------------------------------------------------------------------
\begin{prop}\label{locally_decomposable}
 Let $(A,[,]_{A},a)$ be a Lie algebroid over $M$ equipped with a Nambu structure $\Pi$ of order $n$ with $n\geq 3$. Then for any $f \in  C^{\infty}(M),~ \Pi (d_Af)$ is locally decomposable. In particular,
if $\Gamma A^\ast$ is locally spanned by elements of the form $d_{A}f$ where $f\in C^\infty(M)$ then $\Pi$ is locally
decomposable.
\end{prop}
%-----------------------------------------------------------------------------------------------------------

Let $(A,[,]_{A},a)$ be a Lie algebroid over $M$ with a Nambu structure $\Pi$ of order $n$.
Then $\Pi$ induces a morphism of vector bundles $\Pi_k\colon\Lambda^k A^\ast\longrightarrow \Lambda^{n-k}A$ for $k\leq n$
defined by
\[
 \Pi_k(\alpha)= i_\alpha\Pi(x)\hspace*{2mm}\text{for all}\hspace*{2mm}\alpha\in\Lambda^k A^\ast_x
\]
\noindent
Hence we have a $C^{\infty}(M)$-linear map $\Pi_k\colon \Gamma (\Lambda^k A^\ast)\longrightarrow \Gamma (\Lambda^{n-k}A)$ given by
\[
 \Pi_k(\alpha):=\Pi\alpha = i_{\alpha} \Pi \hspace*{2mm}\text{for all}\hspace*{2mm}\alpha\in\Gamma (\Lambda^k A^\ast).
\]

\begin{remark}
Note that the induced map $\bar{\Pi}_k\colon \frac{\Gamma (\Lambda^k A^\ast)}{\mbox{ker}~\Pi_k}\longrightarrow \Pi_k(\Gamma (\Lambda^k A^\ast))$ given by
$$\bar{\Pi}_k([\alpha]):=\Pi\alpha = i_{\alpha} \Pi \hspace*{2mm}\text{for all}\hspace*{2mm}\alpha\in\Gamma (\Lambda^k A^\ast)$$ is a $C^{\infty}(M)$-linear module isomorphism where $[\alpha] \in \frac{\Gamma (\Lambda^k A^\ast)}{\mbox{ker}~\Pi_k}.$
\end{remark}

\begin{lemma}\label{subbundle}
Let $(A,[,]_{A},a)$ be a Lie algebroid over $M$ equipped with a regular Nambu structure $\Pi$ of order $n$. Assume further that $\Gamma A^\ast$ is locally spanned by elements of the form $d_{A}f$ where $f\in C^\infty(M)$. Then

$$ \mbox{dim}_\mathbb{R}\Pi_{n-1}(\Lambda^{n-1} A^\ast_x)= n,~~x \in M.$$

Let  $\mathcal D = \bigcup_{x \in M}{\mathcal D}_x$ of $A$ where ${\mathcal D}_x = \Pi_{n-1}(\Lambda^{n-1}A_x^\ast).$ Thus $\mathcal D$ is a subbundle of $A$.
\end{lemma}
\begin{proof}
Let  $x\in M$. Then by Lemma \ref{linearly_independent}, we can find a trivializing neighbourhood $U$ of $x$ and linearly independent local sections $X_1, \ldots, X_n$ on $U$ such that 
\[
 \Pi\mid_U= X_1\wedge\cdots\wedge X_n.
\]

Now extend $X_1,\ldots, X_n$ to linearly independent sections $X_1, \ldots, X_n, X_{n+1}, \ldots, X_m$ on $U$, $m$ being the rank of the vector bundle $A$, so that at each point $y\in U$, $X_1(y), \ldots,  X_m(y)$ is a basis of $A_y$.

Let $\alpha\in \Lambda^{n-1}A^\ast_x$. Then
\[
 \alpha= \displaystyle{\sum_{1\leq i_1<\cdots<i_{n-1}\leq m}}\alpha_{i_1\ldots i_{n-1}}
X_{i_1}^\ast(x)\wedge\cdots\wedge X_{i_{n-1}}^\ast(x)\hspace*{2mm}\text{for some} \hspace*{2mm}
\alpha_{i_1\ldots i_{n-1}}\in\mathbb{R}
\]
\noindent After rearranging terms and renaming, we have
\[
 \alpha= \displaystyle{\sum_{i=1}^n}\alpha_i X_1^\ast(x)\wedge\cdots\wedge
\widehat{X_i^\ast(x)}\wedge\cdots\wedge X^\ast_n(x)
+ \alpha^\prime\hspace*{2mm}\text{where}\hspace*{2mm} \Pi_{n-1}(\alpha^\prime)= 0.
\]
\noindent Applying $\Pi_{n-1}$ on $\alpha$ we get,
\[
 \Pi_{n-1}(\alpha)= \displaystyle{\sum_{i=1}^n}(-1)^{n-i}\alpha_i X_i(x)
\]
\noindent Consider $\beta_i\in\Lambda^{n-1}A^\ast_x$ where
$\beta_i= (-1)^{n-i} X_1^\ast(x)\wedge\cdots\wedge\widehat{X_i^\ast(x)}\wedge\cdots\wedge X^\ast_n(x)$.\\
Then $\Pi_{n-1}(\beta_i)= X_i(x)$. Hence $\text{dim}_\mathbb{R}\Pi_{n-1}(\Lambda^{n-1} A^\ast_x)\geq n$.\\
Clearly, $\text{dim}_\mathbb{R}\Pi_{n-1}(\Lambda^{n-1} A^\ast_x)\leq n$.\\
Hence the lemma.
\end{proof}
\begin{remark}\label{orthogonal} 
Let $(A,[,]_{A},a)$ be a Lie algebroid over a smooth  manifold $M$ equipped with a regular Nambu structure $\Pi$ of order $n\leq m= ~rank~A.$  Assume that $\Gamma (A^\ast)$ is locally spanned by elements of the form $d_Af,~~f \in C^{\infty}(M).$ 
\begin{enumerate}
\item For all $k,~ k \leq n$, $\mbox{Ker}~\Pi_k$ is a vector subbundle of $\Lambda^k A^\ast \rightarrow M$ of rank 
$$\left ( \begin{array}{c}
m\\k
\end{array} \right) - \left ( \begin{array}{c}
n\\k
\end{array} \right).$$ 
Similarly, $\Pi_k(\Lambda^kA^\ast)$ is a subbundle of $\Lambda^{n-k}A \rightarrow M$ of rank $\left ( \begin{array}{c}
n\\k
\end{array} \right).$

\item Since $M$ is paracompact, we may assume that the vector bundle $A$ is Euclidean. Let ${\mathcal D}^0(x)$ be the annihilator of the subspace $\mathcal D (x)$ of $A_x$. Thus ${\mathcal D}^0(x)$ consists of all elements $u^\ast \in A^\ast_x$ such that $u^\ast(v) =0$ for all $v\in {\mathcal D} (x).$ Note that ${\mathcal D}^0(x) = \mbox{Ker}~(\Pi_1|A^\ast_x)$ for all $x \in M$. An $(n-k)$-multisection $P$ is orthogonal to $\mathcal D$ if $\iota_{u^\ast}P(x) =0,$ for all $u^\ast \in {\mathcal D}^0(x), ~~x \in M.$ Then the space of smooth sections of the vector bundles $\Pi_k(\Lambda ^kA^\ast) \rightarrow M$ are precisely the set of all $(n-k)$-multisections of $A$ that are orthogonal to $\mathcal D.$
\end{enumerate}
\end{remark}
%-----------------------------------------------------------------------------------------------------------

\noindent We have the following proposition \cite{wade}.
\begin{prop}\label{wade-leibniz-algebroid}
Let $(A,[,]_{A},a)$ be a Lie algebroid over $M$ with a Nambu structure $\Pi$ of order $n$.
In that case, the triplet
$(\mathscr{A}= \Lambda^{n-1}A^\ast,\ll,\gg_\ast,a\circ \Pi_{n-1})$ determines a Leibniz algebroid, where
$\ll\cdot,\cdot\gg_\ast$ is defined by
\[
 \ll \alpha,\beta\gg_\ast= \mathcal{L}_{\Pi\alpha}\beta- \iota_{\Pi\beta}d_{A}\alpha,\hspace*{3mm}
\mbox{for any}\hspace*{1mm} \alpha,\beta\in \Gamma \mathscr{A}.
\]
\end{prop}

%-----------------------------------------------------------------------------------------------------------------
\begin{lemma}
 Let $(A,[,]_{A},a)$ be a Lie algebroid over $M$ with a Nambu structure $\Pi$ of order $n$.
Let $\eta_1,\eta_2\in \Gamma \Lambda^{n-1} (A)$. Then
\[
[\Pi\eta_1,\Pi\eta_2]_{A}= [\Pi\eta_1,\Pi]_{A}\eta_2+ \Pi(\mathcal{L}_{\Pi\eta_1}\eta_2)
\]
\end{lemma}
\begin{proof}
It is enough to verify the equality for
$\Pi= X_1\wedge X_2\wedge\cdots\wedge X_n$, $\eta_1= Y_1\wedge\cdots\wedge Y_{n-1}$
and $\eta_2= Z_1\wedge\cdots\wedge Z_{n-1}$.\\
Also, by induction, it is enough to verify this for $n=2$, that is, for $\Pi= X_1\wedge X_2$ where
$X_1, X_2\in\Gamma A$ and $\eta_1, \eta_2\in\Gamma A^\ast$.\\
Consider
\[
[\Pi\eta_1,\Pi\eta_2]_{A}= [\Pi\eta_1,\iota_{\eta_2}(X_1\wedge X_2)]_{A}
\]
\noindent We know that for $\theta\in\Gamma A^\ast$ and $\xi_1,\xi_2\in\Gamma (\Lambda^\bullet A)$,
\[
 \iota_\theta(\xi_1\wedge\xi_2)= \iota_\theta(\xi_1)\wedge\xi_2+ (-1)^{\text{deg}(\xi_1)}\xi_1\wedge \iota_\theta(\xi_2)
\]
\noindent Hence we have
\begin{align*}
 [\Pi\eta_1,\Pi\eta_2]_{A}=& [\Pi\eta_1, \iota_{\eta_2}(X_1)\wedge X_2- X_1\wedge \iota_{\eta_2}(X_2)]_{A}\\
           =& [\Pi\eta_1, \iota_{\eta_2}(X_1)\wedge X_2]_{A}- [\Pi\eta_1,X_1\wedge \iota_{\eta_2}(X_2)]_{A}
\end{align*}
\noindent By the definition of Gerstenhaber bracket $[,]_{A}$ on $\Gamma (\Lambda^\bullet A)$, we have
\[
[\Pi\eta_1,\Pi\eta_2]_{A}\] 
\[= [\Pi\eta_1,\iota_{\eta_2}(X_1)]_{A}\wedge X_2+ \iota_{\eta_2}(X_1)[\Pi\eta_1,X_2]_{A}
            -[\Pi\eta_1, X_1]_{A}\wedge \iota_{\eta_2}(X_2)- X_1\wedge[\Pi\eta_1,\iota_{\eta_2}(X_2)]_{A}
\]
$$= a(\Pi\eta_1)(\iota_{\eta_2}(X_1))\wedge X_2+ \iota_{\eta_2}(X_1)[\Pi\eta_1,X_2]_{A}
            -[\Pi\eta_1, X_1]_{A}\wedge \iota_{\eta_2}(X_2)- X_1\wedge a(\Pi\eta_1)(\iota_{\eta_2}(X_2)).$$
Notice that $[\Pi\eta_1,\iota_{\eta_2}(X_1)]_{A}= a(\Pi\eta_1)(\iota_{\eta_2}(X_1))$.        

Similarly, we get
$$[\Pi\eta_1, \Pi]_{A}\eta_2= \iota_{\eta_2}(X_1)[\Pi\eta_1,X_2]_{A}- X_1\wedge \iota_{\eta_2}([\Pi\eta_1, X_2]_{A}).$$
Now consider
\begin{align*}
\Pi(\mathcal{L}_{\Pi\eta_1}\eta_2)=&\iota_{\mathcal{L}_{\Pi\eta_1}\eta_2}\Pi=\iota_{\mathcal{L}_{\Pi\eta_1}\eta_2}(X_1\wedge X_2)\\
       =&\big(\iota_{\mathcal{L}_{\Pi\eta_1}\eta_2} X_1\big)\wedge X_2- X_1\wedge\big(\iota_{\mathcal{L}_{\Pi\eta_1}\eta_2} X_2\big)
\end{align*}
\noindent By the definition of Lie derivative and contraction, we get
\begin{align*}
 \iota_{\mathcal{L}_{\Pi\eta_1}\eta_2} X_1= & \big(\mathcal{L}_{\Pi\eta_1}\eta_2 \big)(X_1)\\
                =& a(\Pi\eta_1)(\eta_2(X_1))- \eta_2([\Pi\eta_1, X_1]_{A})\\
                =& a(\Pi\eta_1)(\eta_2(X_1))- \iota_{\eta_2}([\Pi\eta_1, X_1]_{A})
\end{align*}
\noindent
By similar arguments we get,
$$\Pi(\mathcal{L}_{\Pi\eta_1}\eta_2)= a(\Pi\eta_1)(\eta_2(X_1))\wedge X_2- \iota_{\eta_2}([\Pi\eta_1, X_1]_{A})\wedge X_2$$
                  $$-X_1\wedge a(\Pi\eta_1)(\eta_2(X_2))+ X_1\wedge \iota_{\eta_2}([\Pi\eta_1, X_2]_{A}).$$

\noindent By combining the above facts we get the result.
\end{proof}
%-----------------------------------------------------------------------------------------------------------

\begin{corollary}\label{morphism-leibalgd}
Let $(A,[,]_{A},a)$ be a Lie algebroid over $M$ with a Nambu structure $\Pi$ of order $n$. Then the vector bundle map $\Pi_{n-1}\colon\Lambda ^{n-1}A^\ast \longrightarrow A$ given by
$$\alpha \mapsto \Pi \alpha, ~~\alpha \in \Gamma \Lambda^{n-1}A^\ast $$ as defined above is a morphism of Leibniz algebroid. 
\end{corollary}
\begin{proof}
Observe that
$$[\Pi\alpha,\Pi\beta]_{A}= [\Pi\alpha,\Pi]_{A}\beta + \Pi(\iota_{\Pi\beta}d_A\alpha)  + \Pi(\mathcal{L}_{\Pi\alpha}\beta- \iota_{\Pi\beta}d_A\alpha)$$
$$=\Pi(\mathcal{L}_{\Pi\alpha}\beta- \iota_{\Pi\beta}d_A\alpha)= \Pi (\ll\alpha, \beta \gg).$$
\end{proof}

The following is a consequence of the above facts. 
\begin{corollary}
The subbundle $\Pi_{n-1}(\Lambda^{n-1}A^\ast)$ is a Lie sub algebroid of $A$.
\end{corollary}
\begin{proof} Let $\alpha,\beta\in \Gamma (\Pi_{n-1}(\Lambda^{n-1}A^\ast))$.\\
Note that
\begin{align*}
 [\Pi_{n-1}\alpha, \Pi_{n-1}\beta]_{A} =& [\Pi\alpha, \Pi]_{A}+ \Pi(\mathcal{L}_{\Pi\alpha}\beta)\\
   =& [\Pi\alpha,\Pi]_{A}\beta+ \Pi(\iota_{\Pi\beta}d_{A}\alpha)+ \Pi(\mathcal{L}_{\Pi\alpha}\beta-\iota_{\Pi\beta}d_{A}\alpha)\\
   =& \Pi(\mathcal{L}_{\Pi\alpha}\beta- \iota_{\Pi\beta}d_{A}\alpha)\\
   =& \Pi_{n-1}(\ll \alpha,\beta\gg_\ast).
\end{align*}
Hence the result.
\end{proof}
%-------------------------------------------------------------------------------------------------------------

Let $(A,[,]_{A},a)$ be a Lie algebroid over $M$ with a Nambu structure $\Pi$ of order $n$. Recall that then $M$ is a Nambu-Poisson manifold of order $n$ with the Nambu-Poisson structure on $M$ being the push-forward $\Lambda = a \circ \Pi$ of $\Pi$ via the anchor $a$. In other words, 
$$\Lambda (df_1, \ldots , df_n) = \{f_1, \ldots , f_n\}_{\Pi} = \Pi (d_Af_1, \ldots , d_Af_n), ~f_1, f_2, \ldots , f_n\in C^{\infty}(M).$$ Moreover, $\Lambda $ is locally decomposable.

We have the following result.

\begin{lemma}\label{3.3} Assume that $\Gamma A^\ast$ is locally spanned by elements of the form $d_{A}f$ where $f\in C^\infty(M)$. Then
$$\mathcal {L}_{a\circ \Pi(\alpha)}\Lambda = (-1)^n(\iota_{d_A\alpha}\Pi)\Lambda~~ \mbox{for all}~~ \alpha \in \Gamma\Lambda^{n-1}A^\ast.$$
\end{lemma}
\begin{proof} It is enough to check the equality for regular points. Let $x\in M$ be a regular point of $\Lambda$. There exists a coordinate chart $(U, x_1, \ldots , x_n, x_{n+1}, \ldots , x_m)$ around $x$ such that $\Lambda/ U$ can be expressed as 
$$\Lambda \mid_U = \frac{\partial}{\partial x_1}\wedge \cdots \frac{\partial}{\partial x_n}.$$ We may assume that on $\alpha/U = fd_Af_1\wedge d_Af_2 \wedge \cdots \wedge d_Af_{n-1},$ where $f,~ f_i \in C^{\infty}(M).$ Consider the $(n-1)$-form $\alpha_0 = fdf_1\wedge \cdots \wedge df_{n-1}$ on $U$.
Note that 
$$\iota_{\alpha_0}\Lambda = f\iota_{(df_1\wedge df_2 \wedge \cdots \wedge df_{n-1})}\Lambda = f\Lambda (df_1\wedge df_2 \wedge \cdots \wedge df_{n-1}).$$
Observe that $\Lambda (df_1\wedge df_2 \wedge \cdots \wedge df_{n-1})(f) = \Lambda (df_1\wedge df_2 \wedge \cdots \wedge df_{n-1}\wedge df) \\= \Pi(d_Af_1\wedge d_Af_2 \wedge \cdots \wedge d_Af_{n-1}\wedge d_Af)\\= (d_Af)\Pi(d_Af_1\wedge d_Af_2 \wedge \cdots \wedge d_Af_{n-1})=a(\Pi(d_Af_1\wedge d_Af_2 \wedge \cdots \wedge d_Af_{n-1}))(f),$ \\for any $f \in C^{\infty}(M).$
Therefore,
$$\iota_{\alpha_0}\Lambda = fa\circ\Pi(d_Af_1\wedge d_Af_2 \wedge \cdots \wedge d_Af_{n-1}) = a\circ \Pi(fd_Af_1\wedge d_Af_2 \wedge \cdots \wedge d_Af_{n-1}) = a\circ \Pi(\alpha).$$
On the other hand,
$$\iota_{d\alpha_0}\Lambda = \Lambda(df\wedge df_1\wedge \cdots \wedge df_{n-1}) = \Pi (d_Af\wedge d_Af_1\wedge \cdots \wedge d_Af_{n-1})=\iota_{(d_Af\wedge d_Af_1\wedge \cdots \wedge d_Af_{n-1})}\Pi =\iota_{d_a\alpha}\Pi.$$ The proof is now complete by \cite[Equation $(3.3)$]{ilmp}.
\end{proof} 
\begin{remark}
Given a multisection $\Pi \in \Gamma(\Lambda^q A)$ of a Lie algebroid $A$ over $M$, we have a $\mathbb R$-multilinear alternating map
$$\Pi : \underbrace{C^{\infty}(M) \times \cdots \times C^{\infty}(M)}_{q\hspace*{1mm}\text{times}} \longrightarrow C^{\infty}(M)$$ defined by 
$$\Pi(f_1, \ldots ,f_q): =  \Pi(d_Af_1, \ldots, d_Af_q)$$ which satisfies the derivation property in each argument. The converse is true when $\Gamma A^\ast$ is locally spanned by elements of the form $d_Af,~ f\in  C^{\infty}(M).$ Because, in that case, we may define a multisection $\Pi \in \Gamma (\Lambda^q A)$ by 
$$\langle \Pi, d_A f_1 \wedge \cdots \wedge d_Af_q\rangle : = \Pi(f_1, \ldots ,f_q).$$ As in the proof of Lemma $1.2.2$ \cite{dz}, we need only to check that the value of $\Pi(f_1, \ldots ,f_q)$ at a point $x \in M$ depends only on the value of $d_Af_1, \ldots ,d_Af_q$ at $x$. Note that for any $f \in C^\infty(M),$ $d_Af$ is given by $d_Af(X) = a(X)(f), ~X\in \Gamma A.$ Thus, if $d_Af$ is zero at a point $x\in M$, then on a neighbourhood of $x$, we may express $f$ in the form $f = c +\sum_ig_ih_i$ where $c$ is a constant and $g_i,~h_i$ are smooth functions which vanish at $x$. The rest of the proof is similar to the proof of \cite[Lemma $1.2.2$]{dz}.
\end{remark}

Let $(A,[,]_{A},a)$ be a Lie algebroid over a smooth manifold $M$ equipped with a Nambu structure $\Pi$ of order $n$, $n\leq m =~rank ~A$. Assume that the vector bundle $A$ is orientable. Let $\nu \in \Gamma(\Lambda^{m} A^\ast)$ be the nowhere vanishing form representing the orientation. We will call it the {\it orientation form}. Further assume that  $\Gamma A^\ast$ is locally spanned by elements of the form $d_{A}f$
where $f \in C^\infty(M)$ and $d_{A}$ denote coboundary operator for the Lie algebroid cohomology of $A$ with trivial coefficients.

Consider the mapping
\[
\mathcal{M}^{\nu}\colon \underbrace{C^\infty(M)\times\cdots\times C^\infty(M)}_{n-1\hspace*{1mm}\text{times}}\longrightarrow
C^\infty(M)
\]
\noindent given by
\[
\mathcal{L}_{\Pi(d_{A}f_1\wedge\cdots\wedge d_{A}f_{n-1})}\nu
= \mathcal{M}^{\nu}(f_1,\ldots,f_{n-1})\nu
\]
 \noindent where $f_1,\ldots,f_{n-1}\in C^\infty(M)$ for all $i=1,\ldots,n-1$. 
Then
$\mathcal{M}^{\nu}$ defines an $(n-1)$ $\mathbb R$-multi-linear alternating map and satisfies Leibniz rule in each argument with respect to usual product of functions. Then by the above remark, $\mathcal{M}^{\nu}$ arises from an $(n-1)$-multisection $\widetilde{\mathcal{M}}_\Pi^\nu \in\Gamma (\Lambda^{n-1}A),$ so that 
\[
 \mathcal{M}^\nu(f_1,\ldots,f_{n-1})=
\langle\widetilde{\mathcal{M}}_\Pi^\nu,d_{A}f_1\wedge\cdots\wedge d_{A}f_{n-1}\rangle.
\]

%----------------------------------------------------------------------------------------------------------
\begin{lemma}\label{formula-lie-derivative}
 For all $\alpha\in\Gamma (\Lambda^{n-1}A)$,
\begin{equation}\label{relation_L_i_modular_class}
 \mathcal{L}_{\Pi\alpha}\nu=
\big(\iota_\alpha\widetilde{\mathcal{M}}_\Pi^\nu+ (-1)^{n-1}\iota_{d_{A}\alpha}\Pi \big)\nu
\end{equation}
\end{lemma}
\begin{proof}
We will prove the equation for $\alpha$ of the form
$$fd_{A}f_1\wedge\cdots\wedge d_{A}f_{n-1},~~ \mbox{where}~~f,f_1,\ldots,f_{n-1}\in C^\infty(M).$$
Consider 
\begin{align*}
\mathcal{L}_{\Pi\alpha}\nu
&= \mathcal{L}_{\Pi(f d_{A}f_1\wedge\cdots\wedge d_{A}f_{n-1})}\nu\\
&= f\mathcal{L}_{\Pi(d_{A}f_1\wedge\cdots\wedge d_{A}f_{n-1})}\nu
+ d_{A}f\wedge \iota_{\Pi(d_{A}f_1\wedge\cdots\wedge d_{A}f_{n-1})}\nu\\
&= f \mathcal{M}^\nu(f_1,\ldots,f_{n-1})\nu
+ d_{A}f\wedge \iota_{\Pi(d_{A}f_1\wedge\cdots\wedge d_{A}f_{n-1})}\nu
\end{align*}
So, we have
\begin{equation}\label{l-pi-1}
\mathcal{L}_{\Pi\alpha}\nu
 = f \mathcal{M}^\nu(f_1,\ldots,f_{n-1})\nu
+ d_{A}f\wedge \iota_{\Pi(d_{A}f_1\wedge\cdots\wedge d_{A}f_{n-1})}\nu.
\end{equation}
\noindent Now, 
\begin{align*}
f \mathcal{M}^\nu(f_1,\ldots,f_{n-1})
&= f \langle\widetilde{\mathcal{M}}_\Pi^\nu,d_{A}f_1\wedge\cdots\wedge d_{A}f_{n-1}\rangle\\
&= f\iota_{d_{A}f_1\wedge\cdots\wedge d_{A}f_{n-1}} \widetilde{\mathcal{M}}_\Pi^\nu\\
&= \iota_{fd_{A}f_1\wedge\cdots\wedge d_{A}f_{n-1}} \widetilde{\mathcal{M}}_\Pi^\nu\\
&= \iota_\alpha \widetilde{\mathcal{M}}_\Pi^\nu, 
\end{align*}
that is,
\begin{equation}\label{l-pi-2}
 f \mathcal{M}^\nu(f_1,\ldots,f_{n-1})= \iota_\alpha \widetilde{\mathcal{M}}_\Pi^\nu.
\end{equation}
\noindent
Note that $d_Af \wedge \nu = 0$ and hence $\iota_{\Pi(d_{A}f_1\wedge\cdots\wedge d_{A}f_{n-1})}(d_Af \wedge \nu)=0.$ Since
\[
 \iota_X(\omega\wedge\tau)= \iota_X(\omega)\wedge\tau+ (-1)^k\omega\wedge \iota_X(\tau),~~\omega \in \Gamma (\Lambda^kA^\ast),
\]
we obtain
\begin{align*}
d_{A}f\wedge \iota_{\Pi(d_{A}f_1\wedge\cdots\wedge d_{A}f_{n-1})}\nu
&= \big( \iota_{\Pi(d_{A}f_1\wedge\cdots\wedge d_{A}f_{n-1})} d_{A}f \big)\wedge \nu\\
&= \big( \iota_{\Pi(d_{A}f_1\wedge\cdots\wedge d_{A}f_{n-1})} d_{A}f \big) \nu\\
&= \langle \Pi(d_{A}f_1\wedge\cdots\wedge d_{A}f_{n-1}), d_{A}f \rangle \nu\\
&= \langle \iota_{(d_{A}f_1\wedge\cdots\wedge d_{A}f_{n-1})}\Pi, d_{A}f \rangle \nu\\
&= \Pi(d_{A}f_1\wedge\cdots\wedge d_{A}f_{n-1}\wedge d_{A}f)\nu\\
&= (-1)^{n-1} \Pi(d_{A}f\wedge d_{A}f_1\wedge\cdots\wedge d_{A}f_{n-1})\nu\\
&= (-1)^{n-1} \Pi(d_{A}(f\wedge d_{A}f_1\wedge\cdots\wedge d_{A}f_{n-1}))\nu\\
&= (-1)^{n-1} \iota_{d_{A}(f\wedge d_{A}f_1\wedge\cdots\wedge d_{A}f_{n-1})}\Pi\nu.
\end{align*}
Hence, we get
\begin{equation}\label{l-pi-3}
 d_{A}f\wedge \iota_{\Pi(d_{A}f_1\wedge\cdots\wedge d_{A}f_{n-1})}\nu
= (-1)^{n-1} i_{d_{A}\alpha}\Pi\nu
\end{equation}
\noindent Thus $(3.1)$ follows from $(3.2)-(3.4)$.
\end{proof}

%---------------------------------------------------------------------------------------------------------------------
We need the following lemma.
\begin{lemma}
For all $\alpha,\beta\in \Gamma ( \Lambda^{n-1}A)$,
\begin{equation}\label{crucial}
 \iota_{d_A\ll \alpha,\beta\gg_\ast}\Pi=
((a\circ\Pi)(\alpha))\iota_{d_A\beta}\Pi- ((a\circ\Pi)(\beta))\iota_{d_A\alpha}\Pi.
\end{equation}
\end{lemma}
\begin{proof}
Let $\Lambda$ be the Nambu-Poisson structure on $M$ induced by $\Pi$. Then, by Lemma \ref{3.3} we have
\begin{align*}
 (\iota_{d_A \ll \alpha,\beta\gg_\ast}\Pi)\Lambda
&= (-1)^n \mathcal{L}_{(a\circ\Pi)(\ll \alpha,\beta\gg_\ast )}\Lambda\\
&= (-1)^n \mathcal{L}_{[(a\circ\Pi)(\alpha),(a\circ\Pi)(\beta)]_A}\Lambda\\
&= (-1)^n \mathcal{L}_{(a\circ\Pi)(\alpha)}\mathcal{L}_{(a\circ\Pi)(\beta)}\Lambda
   - (-1)^n \mathcal{L}_{(a\circ\Pi)(\beta)}\mathcal{L}_{(a\circ\Pi)(\alpha)}\Lambda\\
&= (-1)^n \mathcal{L}_{(a\circ\Pi)(\alpha)}((-1)^n (\iota_{d_A\beta}\Pi)\Lambda)
   - (-1)^n \mathcal{L}_{(a\circ\Pi)(\beta)}((-1)^n (\iota_{d_A\alpha}\Pi)\Lambda)\\
&= \mathcal{L}_{(a\circ\Pi)(\alpha)}((\iota_{d_A\beta}\Pi)\Lambda)
   - \mathcal{L}_{(a\circ\Pi)(\beta)}((\iota_{d_A\alpha}\Pi)\Lambda)\\
&= (\iota_{d_A\beta}\Pi)\mathcal{L}_{(a\circ\Pi)(\alpha)}\Lambda+ ((a\circ\Pi)(\alpha))(\iota_{d_A\beta}\Pi)\Lambda\\
&   - (\iota_{d_A\alpha}\Pi)\mathcal{L}_{(a\circ\Pi)(\beta)}\Lambda- ((a\circ\Pi)(\beta))(\iota_{d_A\alpha}\Pi)\Lambda\\
&= (\iota_{d_A\beta}\Pi) (-1)^n (\iota_{d_A\alpha}\Pi)\Lambda+ ((a\circ\Pi)(\alpha))(\iota_{d_A\beta}\Pi)\Lambda\\
&   - (\iota_{d_A\alpha}\Pi)(-1)^n (\iota_{d_A\beta}\Pi)\Lambda- ((a\circ\Pi)(\beta))(\iota_{d_A\alpha}\Pi)\Lambda\\
&=  ((a\circ\Pi)(\alpha))(\iota_{d_A\beta}\Pi)\Lambda- ((a\circ\Pi)(\beta))(\iota_{d_A\alpha}\Pi)\Lambda.
\end{align*}
\noindent So, we conclude that
\[
 \iota_{d_A\ll \alpha,\beta\gg_\ast}\Pi=
((a\circ\Pi)(\alpha))\iota_{d_A\beta}\Pi- ((a\circ\Pi)(\beta))\iota_{d_A\alpha}\Pi.
\]
\end{proof}

\begin{thm}\label{nambu-modular-class}
(1) The mapping
$ \mathcal{M}_\Pi^\nu\colon\Gamma (\Lambda^{n-1}A)\longrightarrow C^\infty(M)$ given by
$\alpha\mapsto \iota_\alpha \widetilde{\mathcal{M}}_\Pi^\nu$
defines a $1$-cocycle in the cohomology complex associated to the Leibniz algebroid
$$(\mathscr{A}=\Lambda^{n-1}A^\ast,\ll,\gg_\ast,a\circ\Pi_{n-1}).$$
(2) The cohomology class $\mathcal{M}_\Pi= [\mathcal{M}_\Pi^\nu]\in \mathcal{H}L^1(\mathscr{A})$ does not depend on
the chosen orientation form.
\end{thm}
\begin{proof}(1) We need to prove that for all $\alpha,\beta\in \Gamma(\mathscr{A}),
d_{\mathscr{A}}\mathcal{M}_\Pi^\nu(\alpha,\beta)= 0$. \\
Recall that for all $\alpha,\beta\in \Gamma(\mathscr{A}),$
\begin{align*}
 (d_{\mathscr{A}}\mathcal{M}_\Pi^\nu)(\alpha,\beta)
&= (a\circ\Pi)(\alpha)\mathcal{M}_\Pi^\nu(\beta)- (a\circ\Pi)(\beta)\mathcal{M}_\Pi^\nu(\alpha)
- \mathcal{M}_\Pi^\nu(\ll \alpha,\beta\gg_\ast)\\
&= (a\circ\Pi)(\alpha) \iota_\beta\widetilde{\mathcal{M}}_\Pi^\nu- (a\circ\Pi)(\beta)\iota_\alpha\widetilde{\mathcal{M}}_\Pi^\nu
- \iota_{\ll \alpha,\beta\gg_\ast}\widetilde{\mathcal{M}}_\Pi^\nu.
\end{align*}
So it is enough to prove that
\[
\iota_{\ll \alpha,\beta\gg_\ast}\widetilde{\mathcal{M}}_\Pi^\nu
= (a\circ\Pi)(\alpha) \iota_\beta\widetilde{\mathcal{M}}_\Pi^\nu- (a\circ\Pi)(\beta)\iota_\alpha\widetilde{\mathcal{M}}_\Pi^\nu.
\]
\noindent By Lemma \ref{formula-lie-derivative}, Corollary \ref{morphism-leibalgd} and Cartan formula, we get
\begin{align*}
(\iota_{\ll \alpha,\beta\gg_\ast}\widetilde{\mathcal{M}}_\Pi^\nu)\nu
&= \mathcal{L}_{\Pi\ll\alpha,\beta\gg_\ast}\nu+ (-1)^n(\iota_{d_A\ll\alpha,\beta\gg_\ast}\Pi)\nu\\
&= \mathcal{L}_{[\Pi\alpha,\Pi\beta]_A}\nu+ (-1)^n(\iota_{d_A\ll\alpha,\beta\gg_\ast}\Pi)\nu\\
&= \mathcal{L}_{\Pi\alpha}\mathcal{L}_{\Pi\beta}\nu- \mathcal{L}_{\Pi\beta}\mathcal{L}_{\Pi\alpha}\nu
  + (-1)^n(\iota_{d_A\ll\alpha,\beta\gg_\ast}\Pi)\nu\\
&=\mathcal{L}_{\Pi\alpha}\big((\iota_\beta\widetilde{\mathcal{M}}_\Pi^\nu+ (-1)^{n-1}\iota_{d_A\beta}\Pi)\nu\big)\\
&- \mathcal{L}_{\Pi\beta}\big((\iota_\alpha\widetilde{\mathcal{M}}_\Pi^\nu+ (-1)^{n-1}\iota_{d_A\alpha}\Pi)\nu\big)\\
&+ (-1)^n(\iota_{d_A\ll\alpha,\beta\gg_\ast}\Pi)\nu.
\end{align*}
The last equality follows from Equation (\ref{relation_L_i_modular_class}).

\noindent Simplifying further using Equation (\ref{crucial}), we have
\begin{multline*}
 (\iota_{\ll \alpha,\beta\gg_\ast}\widetilde{\mathcal{M}}_\Pi^\nu)\nu=
(\iota_\beta\widetilde{\mathcal{M}}_\Pi^\nu+ (-1)^{n-1}i_{d_A\beta}\Pi)\mathcal{L}_{\Pi\alpha}\nu
+((a\circ\Pi)(\alpha))(\iota_\beta\widetilde{\mathcal{M}}_\Pi^\nu+ (-1)^{n-1}\iota_{d_A\beta}\Pi)\nu\\
-(\iota_\alpha\widetilde{\mathcal{M}}_\Pi^\nu+ (-1)^{n-1}\iota_{d_A\alpha}\Pi)\mathcal{L}_{\Pi\beta}\nu
-((a\circ\Pi)(\beta))(\iota_\alpha\widetilde{\mathcal{M}}_\Pi^\nu+ (-1)^{n-1}\iota_{d_A\alpha}\Pi)\nu\\
+(-1)^n\big(((a\circ\Pi)(\alpha))\iota_{d_A\beta}\Pi- ((a\circ\Pi)(\beta))\iota_{d_A\alpha}\Pi\big)\nu.
\end{multline*}
\noindent After rearranging and canceling and using Equation (\ref{relation_L_i_modular_class}), we have
\begin{multline*}
 (\iota_{\ll \alpha,\beta \gg_\ast}\widetilde{\mathcal{M}}_\Pi^\nu)\nu
= (\iota_\beta\widetilde{\mathcal{M}}_\Pi^\nu+ (-1)^{n-1}\iota_{d_A\beta}\Pi)
\big(\iota_\alpha\widetilde{\mathcal{M}}_\Pi^\nu+ (-1)^{n-1}\iota_{d_{A}\alpha}\Pi \big)\nu
+\big(((a\circ\Pi)(\alpha))\iota_\beta\widetilde{\mathcal{M}}_\Pi^\nu\big)\nu\\
-(\iota_\alpha\widetilde{\mathcal{M}}_\Pi^\nu+ (-1)^{n-1}\iota_{d_A\alpha}\Pi)
\big(\iota_\beta\widetilde{\mathcal{M}}_\Pi^\nu+ (-1)^{n-1}\iota_{d_{A}\beta}\Pi \big)\nu
-\big(((a\circ\Pi)(\beta))\iota_\alpha\widetilde{\mathcal{M}}_\Pi^\nu\big)\nu.
\end{multline*}
\noindent 
Therefore,
\[
  (\iota_{\ll \alpha,\beta\gg_\ast}\widetilde{\mathcal{M}}_\Pi^\nu)\nu=
\big(((a\circ\Pi)(\alpha))\iota_\beta\widetilde{\mathcal{M}}_\Pi^\nu\big)\nu
-\big(((a\circ\Pi)(\beta))\iota_\alpha\widetilde{\mathcal{M}}_\Pi^\nu\big)\nu.
\]
\noindent This proves the first part.

\noindent(2) Let $\nu^\prime$ be another orientation form representing the given orientation of $A$. Then there exists $f\in C^{\infty}(M)$, $f>0$ everywhere such that $\nu^\prime= f\nu$.
Let $\mathcal{M}_\Pi^\prime= [\mathcal{M}_\Pi^{\nu^\prime}]$.
We need to prove that $\mathcal{M}_\Pi^\prime= \mathcal{M}_\Pi$.\\
We have
\begin{align*}
\mathcal{L}_{\Pi\alpha}\nu^\prime
&= \big(\iota_\alpha \widetilde{\mathcal{M}}_\Pi^{\nu^\prime}+ (-1)^{n-1}\iota_{d_A\alpha}\Pi\big)\nu^\prime\\
&= f\big(i_\alpha \widetilde{\mathcal{M}}_\Pi^{\nu^\prime}+ (-1)^{n-1}i_{d_A\alpha}\Pi\big)\nu.
\end{align*}
\noindent On the other hand,
\begin{align*}
 \mathcal{L}_{\Pi\alpha}\nu^\prime
&= \mathcal{L}_{\Pi\alpha} f\nu\\
&= f\mathcal{L}_{\Pi\alpha}\nu+ ((a\circ \Pi)(\alpha)f)\nu\\
&= f\big(\iota_\alpha\widetilde{\mathcal{M}}_\Pi^\nu+ (-1)^{n-1}i_{d_{A}\alpha}\Pi \big)\nu+ ((a\circ \Pi)(\alpha)f)\nu
\end{align*}
\noindent Hence, we have
\[
f (\iota_\alpha \widetilde{\mathcal{M}}_\Pi^{\nu^\prime})\nu= f(\iota_\alpha\widetilde{\mathcal{M}}_\Pi^\nu)\nu+ ((a\circ \Pi)(\alpha)f)\nu.
\]
\noindent Thus we obtain,
\[
 \iota_\alpha \widetilde{\mathcal{M}}_\Pi^{\nu^\prime}= \iota_\alpha\widetilde{\mathcal{M}}_\Pi^\nu+ {\frac{1}{f}}((a\circ \Pi)(\alpha)f).
\]
\noindent Notice that
${\frac{1}{f}}((a\circ \Pi)(\alpha)f)= ((a\circ \Pi)(\alpha))(ln(f))= d_{\mathscr{A}}(ln(f))(\alpha)$.
Hence we have,
\[
 \mathcal{M}^{\nu^\prime}_\Pi= \mathcal{M}^{\nu}_\Pi+ d_{\mathscr{A}}(ln(u)).
\]
This implies  $\mathcal{M}_\Pi^\prime= \mathcal{M}_\Pi$.
\end{proof}
\begin{defn}
 Let $(A,[,]_{A},a)$ be a Lie algebroid over a smooth manifold $M$ equipped with a Nambu structure $\Pi$
of order $n\geq 3$. Let $A$ be orientable as a vector bundle and $\Gamma (A^\ast)$ be locally spanned by elements of the form $d_Af,~~ f\in C^{\infty}(M)$. Then the cohomology class
$\mathcal{M}_\Pi= [\mathcal{M}^{\nu}_{\Pi}] \in \mathcal{H}L^1(\mathscr{A})$ as defined above is called the {\it{the modular class}} of $  A$.
\end{defn}

\section{a duality theorem}
Let $(A,[,]_{A},a)$ be a Lie algebroid over a smooth manifold $M$ equipped with a regular Nambu structure $\Pi$
of order $n$, $3\leq n \leq m,$ $m$ being the rank of the vector bundle $A$. In this section, we first define Nambu cohomology modules of A and give an equivalent formulation of this cohomology which is a variant of foliated cohomology \cite{vais71, vais73} in the present context. Then we define oriented Nambu homology of $A$, when $A$ is oriented as a vector bundle and prove a duality theorem connecting oriented Nambu homology modules and Nambu cohomology modules.  Throughout, in this section, we assume that $\Gamma (A^\ast)$ is locally spanned by elements of the form $d_Af,~ f \in C^{\infty}(M)$.

Recall from Lemma \ref{subbundle} that we have a subbundle $\mathcal D = \bigcup_{x \in M}{\mathcal D}_x$ of $A$ where ${\mathcal D}_x = \Pi_{n-1}(\Lambda^{n-1}A_x^\ast).$ Therefore, the vector bundle morphism  $\Pi_{n-1}\colon \Lambda^{n-1}A^\ast \longrightarrow A$ (we also denote by $\Pi_{n-1}$ the corresponding $C^{\infty}(M)$-linear map $\Gamma (\Lambda^{n-1}A^\ast) \mapsto \Gamma(A)$) induces a vector bundle isomorphism 
$$\overline{\Pi}_{n-1} \colon \frac{\Lambda^{n-1}A^\ast}{\mbox{ker}~\Pi_{n-1}} \longrightarrow \Pi_{n-1}(\Lambda^{n-1}A^\ast).$$ 

Note that $\Gamma ( \frac{\Lambda^{n-1}A^\ast}{\mbox{ker}~\Pi_{n-1}})$ can be identified with $\frac{\Gamma(\Lambda^{n-1}A^\ast)}{\mbox{ker}~\Pi_{n-1}}$ and similarly, the space of smooth sections of $\Pi_{n-1}(\Lambda^{n-1} A^\ast)$ can be identified with $\Pi_{n-1}(\Gamma (\Lambda^{n-1}A^\ast)).$ With the above identifications, the $C^{\infty}(M)$-linear map
$$\overline{\Pi}_{n-1} \colon \frac{\Gamma(\Lambda^{n-1}A^\ast)}{\mbox{ker}~\Pi_{n-1}} \longrightarrow \Pi_{n-1}(\Gamma (\Lambda^{n-1}A^\ast))$$ may be described simply by $\overline{\Pi}_{n-1}([\alpha]) = \Pi_{n-1}(\alpha)$ for $[\alpha] \in \frac{\Gamma(\Lambda^{n-1}A^\ast)}{\mbox{ker}~\Pi_{n-1}}.$ 

Thus, we have an associated Lie algebroid
$$(\frac{\Lambda^{n-1}A^\ast}{\mbox{ker}~\Pi_{n-1}}, \langle~,~\rangle, a\circ \bar{\Pi}_{n-1})$$ where the Lie bracket on the space of smooth sections and the anchor are given respectively as follows.
$$\langle[\alpha], [\beta]\rangle : = [\langle\alpha,\beta\rangle],~~a\circ \bar{\Pi}_{n-1} ([\alpha]) \colon = a\circ \Pi (\alpha),$$
for all $\alpha, ~\beta \in \frac{\Gamma (\Lambda^{n-1}A^\ast)}{\mbox{ker}~\Pi_{n-1}}.$
Moreover, note that $\mathcal D$ is a Lie subalgebroid of $A$ and  
$$\overline{\Pi}_{n-1} \colon \frac{\Lambda^{n-1}A^\ast}{\mbox{ker}~\Pi_{n-1}} \longrightarrow \Pi_{n-1}(\Lambda^{n-1}A^\ast)$$ is a Lie algebroid isomorphism.   

Consider the trivial representation of the Lie algebroid $(\frac{\Lambda^{n-1}A^\ast}{\mbox{ker}~\Pi_{n-1}}, \langle~,~\rangle, a\circ \bar{\Pi}_{n-1})$ given by
$$\frac{\Gamma (\Lambda^{n-1}A^\ast)}{\mbox{ker}~\Pi_{n-1}}\times C^\infty(M) \longrightarrow C^\infty(M),~~([\alpha],f)\mapsto (a\circ \Pi_{n-1}(\alpha))f.$$
\begin{defn}
{\it Nambu cohomology} of the Lie algebroid $(A,[,]_{A},a)$ is by definition the Lie algebroid cohomology of the associated Lie algebroid
$$(\frac{\Lambda^{n-1}A^\ast}{\mbox{ker}~\Pi_{n-1}}, \langle~,~\rangle, a\circ \tilde{\Pi}_{n-1})$$ with trivial representation and will be denoted by $\mathcal{H}^\bullet_{\mbox{N}}(A)$.
\end{defn}
Next we give an equivalent formulation Nambu cohomology. We define a cochain complex as follows. 

Let $(\Gamma (\Lambda^{\bullet}A^\ast), d_A)$ be the cochain complex defining the Lie algebroid cohomology of $A$. Set  
$$\Omega^k(A,\Pi)= \{\alpha\in \Gamma (\Lambda^{k} A^\ast)|\alpha(X_1,\ldots,X_k)=0, \hspace*{2mm}
\mbox{for all} \hspace*{2mm}X_i\in\Pi_{n-1}(\Gamma (\Lambda^{n-1}A))\}.$$
Suppose $\alpha\in\Omega^k(A,\Pi)$, then it is straight forward to verify using Corollary \ref{morphism-leibalgd} that \\  $d_{A}\alpha\in\Omega^{k+1}(A,\Pi)$. Define
$$\Omega^k_{\Pi}(A):= \Gamma(\Lambda^k A^\ast)/\Omega^k(A,\Pi).$$
Then the coboundary operator $d_{A}$ induces a square zero operator
$$\tilde{d}_{A}\colon\Omega^k_{\Pi}(A) \longrightarrow \Omega_{\Pi}^{k+1}(A)$$ where $\tilde{d}_{A}([\alpha])= [d_{A}\alpha]$ for all $[\alpha]\in \Omega_{\Pi}^k(A).$
We shall denote the corresponding cohomology  by $\mathcal{H}_{\Pi}^\bullet(A).$

Define a map $ \Gamma (\Lambda^kA^\ast) \longrightarrow \Gamma(\Lambda ^k {\mathcal D}^\ast)$ given by
$$i^k(\alpha)(X_1, \ldots , X_k) = \alpha ( X_1, \ldots , X_k),~~ \mbox{for}~~\alpha \in \Gamma (\Lambda^k A^\ast),~~ X_i \in \Gamma (\mathcal D).$$ 
Note that this map is surjective, because, we may assume that the vector bundle $A$ is Euclidean and hence every sub bundle admits a compliment. Therefore, for every $f \in \Gamma(\Lambda ^k {\mathcal D}^\ast)$, we may find $\tilde{f} \in \Gamma(\Lambda ^k {\mathcal A}^\ast)$ with the property $i^k(\tilde{f}) = f.$ It is also clear from the definition that the kernel of $i^k$ is $\Omega^k(A,\Pi).$ Therefore, for each $k$, we have a $C^{\infty}(M)$-linear isomorphism $$ \pi^k : \Omega^k_{\Pi}(A) \longrightarrow \Gamma(\Lambda ^k {\mathcal D}^\ast).$$ Moreover, we have $\pi^{k+1}\circ \tilde{d}_A = d_D\circ \pi^k,$ because, for any sections $X_0, X_1, \ldots ,X_k \in \Gamma(\mathcal D)$ and $\alpha \in \Gamma (\Lambda^k A^\ast),$
$d_D\circ \pi^k ([\alpha])(X_0, \ldots , X_k)$\\
$= \Sigma_{i=0}^k (-1)^ia(X_i)\pi^k([\alpha])(X_0, \ldots , \widehat{X_i}, \ldots, X_k)$\\ 
$+ \Sigma_{i<j}(-1)^{i+j}\pi^k([\alpha])([X_i,X_j]_A, \ldots , \widehat{X_i}, \ldots, \widehat{X_j}, \ldots, X_k)$\\
$= \Sigma_{i=0}^k (-1)^ia(X_i)\alpha(X_0, \ldots , \widehat{X_i}, \ldots, X_k)$\\ 
$+ \Sigma_{i<j}(-1)^{i+j}\alpha([X_i,X_j]_A, \ldots , \widehat{X_i}, \ldots, \widehat{X_j}, \ldots, X_k)$\\
$= (d_A\alpha)(X_0, \ldots, X_k)$\\ 
$= \pi^{k+1}([d_a\alpha])(X_0, \ldots, X_k)$\\
$= (\pi^{k+1}\circ \tilde{d}_A)([\alpha])X_0, \ldots, X_k).$

Thus we have the following result.

\begin{prop}
 Let $(A,[,]_A,a)$ be a Lie algebroid with a regular Nambu structure $\Pi$ of order $n$ with $n\geq 3$. Assume that
$\Gamma A^\ast$ is locally spanned by elements of the form $d_Af$ where $f\in C^\infty(M)$. Then the $C^{\infty}(M)$-linear homomorphisms
$$i^k(\alpha)(X_1, \ldots , X_k) = \alpha ( X_1, \ldots , X_k),~~ \mbox{for}~~\alpha \in \Gamma (\Lambda^k A^\ast),~~ X_i \in \Gamma (\mathcal D)$$ 
\noindent induce an isomorphism of complexes $\pi^\bullet\colon (\Omega^\bullet_{\Pi}(A),\tilde{d}_A)\longrightarrow
(\Gamma(\Lambda^{\bullet}{\mathcal D}^\ast), d_D)$. Thus, 
$$\mathcal{H}^\bullet_{\mbox{N}}(A) \cong \mathcal{H}_{\Pi}^\bullet(A).$$
\end{prop}
The following result follows from the definition of regularity of Nambu structure (see Proposition $4.2.$, \cite{illmp} for details).
\begin{prop}\label{kernel-pik}
 Let $(A,[,]_{A},a)$ be a Lie algebroid over a smooth manifold $M$. Let $\Pi$ be a regular Nambu structure
on $A$ of order $n$. Then 
\[
\Omega^k(A, \Pi)= \hspace*{1mm}\text{ker}~\Pi_{k},\hspace*{2mm}\text{for all}\hspace*{3mm} k=1,\ldots,n.
\]
\end{prop}

Next we define the notion of {\it canonical Nambu homology} of a Lie algebroid which is oriented as a vector bundle and equipped with a Nambu structure. Let $A$ be an oriented Lie algebroid with rank $A= m$. Let $\nu$ be the chosen orientation form on $A$. Then we have a vector
bundle isomorphism $\flat_\nu\colon \Lambda^kA\longrightarrow \Lambda^{m-k}A^\ast$ given by
\[
 \flat_\nu(P)= \iota_P\nu \hspace*{2mm}\text{for all}\hspace*{2mm}P\in \Gamma (\Lambda^k A).
\]
\noindent Define an operator $\delta_\nu$ as
\[
 \delta_\nu:= \flat_\nu^{-1}\circ d_{A}\circ \flat_\nu\colon  \Gamma (\Lambda^k A)\longrightarrow \Gamma (\Lambda^{k-1} A)
\]
\noindent Clearly, $\delta_\nu^2= 0$. The homology of the complex $(\Gamma(\Lambda^\bullet A),\delta_\nu)$ is denoted by $\mathcal{H}^\nu_\bullet(A)$.\\
Notice that the cohomology  $\mathcal{H}^\nu_\bullet(A)$ is dual of the Lie algebroid cohomology of $A$ with
trivial coefficients. So, $\mathcal{H}^\nu_\bullet(A)$ does not depend on the choice of orientation form on $A$.
Note that for $X\in\Gamma A$, we have $\mathcal{L}_X\nu= (\delta_\nu(X))\nu.$ This is because,
$$\delta_\nu(X)= (\flat_\nu^{-1}\circ d_{A}\circ \flat_\nu)(X)= \flat_\nu^{-1}\circ d_{A}\circ \iota_X\nu,$$
and since we have $\mathcal{L}_X= d_{A}\circ \iota_X+ \iota_X\circ d_{A}$, it follows that $\delta_\nu(X)= \flat_\nu^{-1}\circ\mathcal{L}_X\nu.$ Therefore,
$$\mathcal{L}_X \nu= \flat_\nu(\delta_\nu(X))= \iota_{\delta_\nu(X)}\nu= (\delta_\nu(X))\nu.$$
In view of this, we may write
\begin{equation}\label{divergence}
\delta_\nu(X)=\hspace*{1mm}\text{div}_\nu(X)\hspace*{2mm}\text{for all}\hspace*{2mm} X\in\Gamma A.
\end{equation}
The proofs of the following two results are analogous to the proofs of Lemma $5.1$ and Proposition $5.2$, \cite {illmp} and hence we omit the details.
\begin{lemma}
 Let $A$ be an oriented Lie algebroid of rank $m$ and $\nu$ be the chosen orientation form on $A$. Then for all
$P\in\Gamma(\Lambda^k A)$ and $X\in\Gamma A$, we have
\[
\mathcal{L}_X\flat_\nu(P)= \flat_\nu(\mathcal{L}_X P)+ \delta_\nu(X)\flat_\nu(P)
\]
\end{lemma}
\begin{prop}\label{ialphadelnu}
 Let $A$ be an oriented Lie algebroid of rank $m$ and $\nu$ be the chosen orientation form on $A$. Then
\[
 \iota_\alpha(\delta_\nu(P))= \delta_\nu(\iota_\alpha P)+ (-1)^k \iota_{d_{A}\alpha}P\hspace*{1mm}\text{for all}\hspace*{1mm}
P\in\Gamma(\Lambda^k A), \alpha \in\Gamma (\Lambda^{k-1}A^\ast).
\]
\end{prop}

We define a subcomplex of $(\Gamma(\Lambda^\bullet A),\delta_\nu)$ as follows.

Let $(A,[,]_{A},a)$ be a Lie algebroid of rank $m$ over a smooth manifold $M$. Let $\Pi$ be a regular Nambu
structure on $A$ of order $n,~ n\leq m$. Then consider the subset of $\Gamma (\Lambda^kA)$ given by
\[
 \mathcal{V}^k(A,\Pi)= \{P\in\Gamma (\Lambda^k A)|~\iota_\alpha P=0\hspace*{1mm}\text{for all}\hspace*{1mm}\alpha\in
\Gamma (A^\ast),\alpha\in\hspace*{1mm}\text{ker}~\Pi_1\},~k\geq 1
\]
\noindent Set $\mathcal{V}^0(A,\Pi)= C^\infty(M)$. Note that by Remark (\ref{orthogonal}), $\mathcal{V}^k(A,\Pi)$ is just the space of all  $k$-multisections of $A$ which are orthogonal to $\mathcal D.$ As a consequence, we have
\begin{lemma}\label{image-pi-n-k}
 Let $(A,[,]_{A},a)$ be a Lie algebroid over $M$ with regular Nambu structure $\Pi$ of order $n$. Then
\[
 \mathcal{V}^k(A,\Pi)= \Pi_{n-k}(\Gamma (\Lambda^{n-k} A^\ast))\hspace*{1mm}\text{for all}\hspace*{1mm}k=1,\ldots,n.
\]
\end{lemma}
Now proceeding as in the proof of Proposition $5.5$, \cite{illmp} and using Proposition (\ref{kernel-pik}) and Proposition (\ref{ialphadelnu}), we obtain
\begin{prop}
Let $(A,[,]_A, a)$ be a Lie algebroid over a smooth manifold $M$ with a regular Nambu structure $\Pi$ of order $n$.
Let $A$ be oriented as a vector bundle and $\nu$ be the chosen orientation form on $A$ representing the orientation. Then
\[
 \delta_\nu(\mathcal{V}^k(A,\Pi))\subseteq \mathcal{V}^{k-1}(A,\Pi)\hspace*{1mm}\text{for all}\hspace*{1mm}
k=1,\ldots,n.
\]
\end{prop}
Therefore, $(\mathcal{V}^{\bullet}(A,\Pi), \delta_{\nu})$ is a subcomplex of $(\Gamma(\Lambda^\bullet A),\delta_\nu)$ and the canonical Nambu homology modules of $A$, denoted by  $\mathcal{H}_\bullet^{\text{canN}}(A)$ are by definition the homology modules of this subcomplex.
\begin{remark}
Note that the canonical Nambu homology of $A$ does not depend on the choice of an orientation form. To see this, let $\nu$ and $\nu^\prime$ be two orientation form on $A$ representing the orientation. Then there exists a nowhere zero function $f\in C^\infty(M)$ such that
$\nu^\prime= f\nu$.\\
Define an isomorphism of $C^\infty(M)$ modules
\begin{gather*}
 \Phi^k\colon \mathcal{V}^k(A,\Pi)\longrightarrow \mathcal{V}^k(A,\Pi)\\
      P\mapsto {\frac{1}{f}} P \hspace*{1mm}\text{for all}\hspace*{1mm}k= 0,\ldots,n.
\end{gather*}
\noindent Then using definition of $\delta_\nu$ and $\delta_{\nu^\prime}$, it is straight forward to prove that
\[
 \delta_{\nu^\prime}\circ\Phi^k= \Phi^{k-1}\circ\delta_\nu
\]
\noindent Therefore, the mapping $\Phi^k$ induces an isomorphism between associated canonical Nambu homologies.
\end{remark}

Let $A$ be a Lie algebroid equipped with a regular Nambu structure $\Pi \in \Gamma (\Lambda^n A)$ of order $n,~ n \leq m =~\mbox{rank}~A$. Assume that $A$ is oriented with $\nu \in \Gamma (\Lambda^m A^\ast)$ being a chosen orientation form representing the orientation on $A$. Then we  have the following results.
%---------------------------------------------------------------------------------------------------------------------

%%%%%%%%%%%%%%%%%%%%%%%%%%%%%%%%%%%%%%%%%%%%%%%%%%%%%%%%%%%%%%%%%%%%%%%%%%%%%%%%%%%%%%%%%%%%%%%%%%
\begin{prop}\label{subcomplex}
The graded vector space $\Pi_\bullet(\Gamma (\Lambda^\bullet A^\ast))= \oplus_{k=0}^n\Pi_k(\Gamma (\Lambda^k A^\ast))$ defines a subcomplex of the chain complex
$(\Gamma (\Lambda ^\bullet A),\delta_\nu)$ if and only if $\widetilde{\mathcal{M}}_\Pi^\nu\in \Pi_1(\Gamma ( A^\ast))$.
\end{prop}
\begin{proof} Let $\alpha \in \Gamma (\Lambda^{n-k-1}A^\ast)$ and $\beta \in \Gamma (\Lambda^k A^\ast)$ be arbitrary. Then, by Lemma (\ref{ialphadelnu})
\begin{align*}
 \iota_\alpha \delta_\nu(\Pi_k \beta)
&= \delta_\nu(\iota_\alpha \Pi_k(\beta))+ (-1)^{n-k}\iota_{d_A\alpha}(\Pi_k(\beta))\\
&= \delta_\nu(\iota_\alpha \iota_\beta\nu)+ (-1)^{n-k}\iota_{d_A\alpha}(\iota_\beta\nu)\\
&= \delta_\nu(\Pi_{n-1}(\beta\wedge\alpha))+ (-1)^{n-k}\Pi_n(\beta\wedge d_A\alpha)\\
&= \iota_\alpha(\iota_\beta\widetilde{\mathcal{M}}_\Pi^\nu+ (-1)^{n-1}\Pi_{k+1}(d_A\beta))
\end{align*}
\noindent Using injectivity of contraction map, we have
\[
 \delta_\nu(\Pi_k \beta)= \iota_\beta\widetilde{\mathcal{M}}_\Pi^\nu+ (-1)^{n-1}\Pi_{k+1}(d_A\beta).
\]
\noindent Hence, we conclude that $\delta_\nu(\Pi_k(\Gamma (\Lambda^k A^\ast)))\subseteq \Pi_{k+1}(\Gamma (\Lambda^{k+1}A^\ast))$ for all
$k=1,\ldots,n$ if and only if $\widetilde{\mathcal{M}}_\Pi^\nu \in \Pi_1(\Gamma (A^\ast))$.
\end{proof}
\begin{prop}
If $\Pi_\bullet(\Gamma (\Lambda^\bullet A^\ast)$ is a subcomplex of $(\Gamma (\Lambda ^\bullet A),\delta_\nu)$, then the homology of this complex is independent of the choice of the orientation form.
\end{prop}
\begin{proof} Let $\nu^\prime$ be another choice of orientation form. Then, there exists a nowhere vanishing real valued function
$f\in C^\infty(M)$ so that $\nu^\prime= f\nu$. Without loss of generality, we may assume that $f>0$. Then,
\begin{gather*}
 \Psi^k \colon \Pi_k(\Omega^k(A)) \longrightarrow \Pi_k(\Omega^k(A))\hspace*{2mm}\text{defined by}
P\mapsto {\frac{1}{f}}P \hspace*{2mm}\text{is an isomorphism.}
\end{gather*}
\noindent Since, $\delta_{\nu^\prime} \circ \Psi^k= \Psi^{k-1}\circ \delta_\nu$, the result follows.
\end{proof}
\begin{prop}
If the modular class of $A$ is null, then the graded vector space $\Pi_\bullet(\Gamma (\Lambda^\bullet A^\ast)$ is a subcomplex of $(\Gamma (\Lambda ^\bullet A),\delta_\nu)$ and
\[
 \bar{\mathcal{H}}_k^{\text{canN}}(A)\cong \mathcal{H}^{n-k}_{\Pi}(A) \hspace*{1mm}\text{for all}\hspace*{1mm}
k=0,\ldots,n.
\]
\noindent where $\bar{\mathcal{H}}_k^{\text{canN}}(A)$ denotes the $k$-th homology of the subcomplex
$\Pi_\bullet(\Gamma (\Lambda^\bullet A^\ast)$.
\end{prop}
\begin{proof} Suppose the modular class of $A$ is null, then there exists $f \in C^\infty(M)$ such that $d_{\mathscr{A}}f = \mathcal{M}_\Pi^\nu.$ Therefore, for all $\alpha \in \Gamma (\Lambda^{n-1} A^\ast),$
$$
\iota_{\alpha}\widetilde{\mathcal{M}}_\Pi^\nu = \mathcal{M}_\Pi^\nu (\alpha) = (d_\mathscr{A}f)\alpha = ((a\circ\Pi_{n-1})(\alpha))f = (a(\iota_\alpha\Pi))f = (d_Af)(\iota_\alpha\Pi)$$
$$= \iota_{d_Af}(\iota_\alpha\Pi) = \iota_\alpha((-1)^{n-1}\iota_{d_Af}\Pi) = \iota_\alpha((-1)^{n-1}\Pi_1(d_Af)).$$
Thus,
$$\widetilde{\mathcal{M}}_\Pi^\nu= \Pi_1((-1)^{n-1}d_Af).$$
Proposition (\ref{subcomplex}) now impiles that $\Pi_\bullet(\Gamma (\Lambda^\bullet A^\ast)$ is a subcomplex of $(\Gamma (\Lambda ^\bullet A),\delta_\nu).$ Next using Proposition (\ref{kernel-pik}), we define isomorphisms of $C^\infty(M)$- modules
$$h_k \colon \Omega_{\Pi}^{n-k}(A)\longrightarrow \Pi_{n-k}(\Gamma (\Lambda^{n-k}A^\ast)),~~ h_k([\alpha]) = e^{-f}\Pi_{n-k}(\alpha)$$ such that $h_k\circ \tilde{d}_A = (-1)^{n-1}\delta_\nu \circ h_{k+1}$. Required isomorphism in cohomologies is induced by the isomorphisms $h_k$.  
\end{proof}
Finally using the above results and  Lemma (\ref{image-pi-n-k}) we obtain the following theorem.
\begin{thm}
If the modular class of $A$ is null then
\[
 \mathcal{H}_{\text{N}}^k(A) \cong \mathcal{H}_{\Pi}^k(A) \cong \mathcal{H}_{n-k}^\text{canN}(A)
\hspace*{2mm}\text{for all}\hspace*{2mm}k= 0,\ldots,n.
\]
\end{thm}
\section{density and the modular class} To define the modular class of a Lie algebroid $A$ equipped with a regular Nambu structure and to prove the duality theorem connecting Nambu cohomology and canonical nambu homology as developed in the previous sections we assumed that $A$ is orientable as a vector bundle. The aim of this last section is show that we can do away with the orientability assumption using the notion of density.

Recall that {\it orientation bundle} $O(A)$ of a smooth vector bundle $A$ over a smooth manifold $M$ is defined as follows \cite{bt}.
\begin{defn}
Let $A$ be a smooth vector bundle over a smooth manifold $M$. Let $\{(U_\alpha, \phi_\alpha)\}$ be a smooth atlas and $h_{\alpha\beta}$ be the corresponding transition functions for $A$. The
{\it{orientation bundle}} of $A$, denoted by $O(A)$, is the real line bundle over $M$ with transition functions $\text{sgn}(J(h_{\alpha\beta}))$, where  $\text{sgn}\colon\mathbb{R}\longrightarrow\mathbb{R}$ is given by
\[
\text{sgn}(x)= 
\begin{cases}
 1 & \hspace*{2mm}\text{for}\hspace*{1mm}x\geq 0,\\ 
 0 & \hspace*{2mm}\text{for}\hspace*{1mm}x= 0,\\
-1 & \hspace*{2mm}\text{for}\hspace*{1mm}x\leq 0
\end{cases}
\]
and $J(h_{\alpha\beta})$ is the Jacobian determinant of the matrix of partial derivatives of $h_{\alpha\beta}$.
\end{defn}  
\begin{defn}
The {\it{density bundle}}, denoted by $D(A)$, of $A$ is 
defined to be $\Lambda^{\text{top}}A^\ast \otimes O(A)$. A section of $D(A)$ is called a {\it{density}} on $A$. 
\end{defn}  
The following is an alternative way to introduce density bundle.
\begin{defn}
Let $V$ be an $n$-dimensional vector space. A density function of $V$ is a function 
\[
\mu\colon\underbrace{V\times \cdots\times V}_{n\hspace*{1mm}\text{-copies}}\longrightarrow \mathbb{R}
\]
\noindent satisfying the following condition: If $T\colon V\longrightarrow V$ is any linear map then
\[
 \mu(T(X_1),\ldots,T(X_n))= |\text{det}(T)|\mu(X_1,\ldots,X_n),~~X_i \in V
 \]
\end{defn}
\noindent Let $D(V)$ denote set of all densities on $V$.

%-------------------------------------------------------------------------------------------------------------------------

\begin{defn}
A {\it{positive}} density on $V$ is a density $\mu$ such that $\mu(X_1,\ldots,X_n)>0$, for all
linearly independent vectors $X_1,\ldots,X_n\in V$.
\end{defn}

Then with the above notations the  density bundle $D(A)$ of $A$ may be described as
\[
 D(A):= \displaystyle{\prod_{x\in M}}D(A_x).
\]

\begin{defn} 
A smooth section of $\mu \in \Gamma(D(A))$ is called a {\it positive density} on $A$ if for all $x \in M$, $\mu_x$ is positive.
\end{defn}

Using a partition unity argument one can show that there always exists a smooth positive density on $A$. More precisely, we have the following lemma.
\begin{lemma}\label{positive density}
Let $A$ be a smooth vector bundle over a smooth manifold $M$. Then there exists a positive density on $A$.
\end{lemma}
\begin{proof} Note that the set of positive elements of $D(A)$ is an open subset whose intersection with each fibre is convex. Then the usual partition unity argument allows us to piece together local positive densities to obtain a global one. 
\end{proof}
%------------------------------------------------------------------------------------------------------------------------
Next we need to introduce {\it twisted Lie algebroid cochain complex}. The construction is similar to the construction of twisted de Rham complex of vector-valued differential forms on smooth manifolds and hence we mention it briefly (see \cite{bt} for details).

Let $A$ be a Lie algebroid and $E$ be a flat vector bundle on a smooth manifold $M$. Recall that a vector bundle $E$ is said to be flat if it admits an atlas $\{(U_\alpha, \phi_\alpha)\}$ relative to which the transition functions are locally constant. We may assume that the atlas $\{(U_\alpha, \phi_\alpha)\}$ on $A$ is induced from an atlas on $M$.  
Let $\alpha \in \Gamma (\Lambda^k A^\ast\otimes E)$ for all $k\in\mathbb{N}.$ Then locally, on an open set $U$, $\alpha= \displaystyle{\sum_i} e_i\otimes f_i,$ 
where $e_i\in\Gamma(U,\Lambda^k A^\ast)$ and $f_i\in\Gamma(U, E)$ and tensor product is over $C^\infty(U)$ and for any vector bundle $E$, $\Gamma(U, E)$ denotes the space of smooth sections of $E$ on $U$. Define $d_A^{{\hspace*{.5mm}\text{t}}}\colon  \Gamma (\Lambda^k A^\ast\otimes E)\longrightarrow \Gamma (\Lambda^{k+1} A^\ast\otimes E) $ locally as
\[
 d_A^{\hspace*{.5mm}\text{t}}\bigg(\displaystyle{\sum_i}e_i\otimes f_i\bigg)= \displaystyle{\sum_i}d_A(e_i)\otimes f_i,
\]
\noindent where $e_i\in\Gamma(U,\Lambda^k A^\ast)$ and $f_i\in\Gamma(U, E)$ and tensor product is over $C^\infty(U)$
for some open subset $U$ of $M$ and then extend $d_A^{\hspace*{.5mm}\text{t}}$ globally using
Leibniz rule and linearity to give a well defined map. Then $(\Gamma (\Lambda^\bullet A^\ast\otimes E), d_A^{\hspace*{.5mm}\text{t}})$ gives the twisted Lie algebroid cochain complex of $A$.
Next just like Definition (\ref{lie-L-i}), we have operators defined as follows.
\begin{defn}
Let $X\in\Gamma A$. Define $\mathcal{L}_X^{\hspace*{.5mm}\text{t}}\colon \Gamma (\Lambda^k A^\ast\otimes E) \longrightarrow  \Gamma (\Lambda^k A^\ast\otimes E)$ locally as
\[
 \mathcal{L}_X^{\hspace*{.5mm}\text{t}}\bigg(\displaystyle{\sum_i}e_i\otimes f_i\bigg)=
\displaystyle{\sum_i}\mathcal{L}_X(e_i)\otimes f_i.
\]
\noindent where $e_i\in\Gamma(U,\Lambda^k A^\ast)$ and $f_i\in\Gamma(U, E)$ and tensor product is over $C^\infty(U)$
for some open subset $U$ of $M$.\\
Similarly,  for $X\in\Gamma(\Lambda^i A),$ we may define $$\iota_X^{\hspace*{.5mm}\text{t}}\colon \Gamma (\Lambda^k A^\ast\otimes E)\longrightarrow  \Gamma (\Lambda^{k-i} A^\ast\otimes E).$$ 
\end{defn}
\begin{remark}
Notice that if $E$ is the trivial line bundle, $d_A^{\hspace*{.5mm}\text{t}},\mathcal{L}_X^{\hspace*{.5mm}\text{t}}$ and $\iota_X^{\hspace*{.5mm}\text{t}}$ are precisely
$d_A,\mathcal{L}_X$ and $\iota_X$, respectively. Also, note that the line bundle $O(A)$ is flat.
\end{remark}

%---------------------------------------------------------------------------------------------------------------------------

\noindent Let $(A,[,]_A,a)$ be a Lie algebroid over a smooth manifold $M$. Let $\omega_x \in \Lambda^i A_x$
and $\tau_x \otimes {o}_x \in \Lambda^j A_x \otimes O(A)_x$, $x \in M$. Define $\omega_x \wedge^{\text{t}} (\tau_x
\otimes {o}_x):= (\omega_x\wedge\tau_x) \otimes o_x$. By extending it linearly, we have a well defined map
\[
 \wedge^\text{t} \colon \Gamma(\Lambda^i A^\ast)\times \Gamma (\Lambda^j A^\ast\otimes O(A)) \longrightarrow \Gamma(\Lambda^{i+j} A^\ast \otimes O(A)).
\]
\noindent The following lemma is easy to prove.
\begin{lemma}
Let $(A,[,]_A,a)$ be a Lie algebroid over a smooth manifold $M$. Let $X\in \Gamma (\Lambda^\bullet A)$. Then for $\omega\in \Gamma (\Lambda^\bullet A^\ast)$ and $\tau \otimes o\in  \Gamma (\Lambda^\bullet A^\ast\otimes O(A)) $, we have
\[
 \iota_X(\omega \wedge^{\text{t}} (\tau\otimes o))= 
\iota_X\omega \wedge^\text{t} (\tau\otimes o)+ (-1)^{\text{deg}(\omega)}\omega\wedge^\text{t}
\iota_X^{\hspace*{.5mm}\text{t}}(\tau\otimes o).
\]
\end{lemma}

%---------------------------------------------------------------------------------------------------------------------------

\begin{defn}
Let $A$ be a Lie algebroid over a smooth manifold $M$. Let $\mu$ be a fixed positive density on $A$. Then,
for any $X\in\Gamma A$, $\text{div}_\mu X\in C^\infty(M)$ is defined by the equation
\[
 \mathcal{L}_X^{\hspace*{.5mm}\text{t}}\mu= \hspace*{2mm}(\text{div}_\mu X) \mu.
\]
\end{defn}

Let $(A,[,]_{A},a)$ be a Lie algebroid over a smooth manifold $M$. Assume that $\Gamma (A^\ast)$ is locally spanned by elements of the form $d_Af,~~f \in C^\infty(M).$ Let $\Pi$ be a Nambu structure on $A$ order $n$. Fix a positive density $\mu$ on $A$. Consider the mapping
\[
\mathcal{M}^{\mu}\colon \underbrace{C^\infty(M)\times\cdots\times C^\infty(M)}_{n-1\hspace*{1mm}\text{times}}\longrightarrow
C^\infty(M)
\]
\noindent given by
\[
\mathcal{L}^{\hspace*{.5mm}\text{t}} _{\Pi(d_{A}f_1\wedge\cdots\wedge d_{A}f_{n-1})}\mu
= \mathcal{M}^{\mu}(f_1,\ldots,f_{n-1})\mu
\]
 \noindent where $f_1,\ldots,f_{n-1}\in C^\infty(M)$ for all $i=1,\ldots,n-1$. 
Then as before
$\mathcal{M}^{\mu}$ defines an $(n-1)$-multisection $\widetilde{\mathcal{M}}_\Pi^\mu \in\Gamma (\Lambda^{n-1}A),$ so that 
\[
 \mathcal{M}^\mu(f_1,\ldots,f_{n-1})=
\langle\widetilde{\mathcal{M}}_\Pi^\mu,d_{A}f_1\wedge\cdots\wedge d_{A}f_{n-1}\rangle.
\]
Then as in Section $3$, we have the following result.
\begin{thm}
(1) The mapping
$ \mathcal{M}_\Pi^\mu\colon\Gamma (\Lambda^{n-1}A)\longrightarrow C^\infty(M)$ given by
$\alpha\mapsto \iota_\alpha \widetilde{\mathcal{M}}_\Pi^\mu$
defines a $1$-cocycle in the cohomology complex associated to the Leibniz algebroid
$$(\mathscr{A}=\Lambda^{n-1}A^\ast,\ll,\gg_\ast,a\circ\Pi_{n-1}).$$
(2) The cohomology class $\mathcal{M}_\Pi= [\mathcal{M}_\Pi^\mu]\in \mathcal{H}L^1(\mathscr{A})$ is independent of the choice of $\mu$.
\end{thm}

\begin{defn}
The cohomology class
$[\mathcal{M}_\Pi]\in \mathcal{H}L^1(\mathscr{A})$ as defined above is called the {\it{the modular class}} of $A$.
\end{defn}

In order to prove a version of duality theorem in the present context, we need to have a version of canonical Nambu homology. To this end, we prove the following result.
\begin{prop}
Let $A$ be a Lie algebroid with a regular Numbu structure $\Pi$ of order $n \leq m = rank~ A$ over a smooth manifold $M$. Let $\mu$ be a positive density on $A$. 
Let $\flat_\mu\colon \Gamma (\Lambda^k A)\longrightarrow \Gamma (\Lambda^{m-k} A^\ast \otimes O(A))$ be defined as $\flat\mu(P)= \iota_P^{\hspace*{.5mm}\text{t}}\mu$ for
all $P\in\Gamma (\Lambda ^k A)$. Then $\flat_\mu$ is an isomorphism.
\end{prop}
\begin{proof}
For dimension reasons, it is enough to prove that $\flat_\mu$ is injective. Let $P\in\Gamma (\Lambda^k A)$ be non-zero. Then $P(x)\neq 0$  for some $x\in M.$
We know $\mu(x)= X_1^\ast(x)\wedge\cdots\wedge X_m^\ast(x)\otimes \mu(x)$ where $X_i^\ast(x)\in A^\ast_x$ and $\mu(x)\in O(A)_x$ is non-zero. Write 
$$P(x)=
\displaystyle{\sum_{1\leq i_1<\cdots<i_k\leq m}}\alpha_{i_1\cdots i_k}X_{i_1}(x)\wedge\cdots\wedge X_{i_k}(x).$$
Without loss of generality, we may assume that $\alpha_{i_1\cdots i_k}\neq 0$ for all tuple $(i_1,\ldots,i_k)$ appearing in
the expression above. Consider 
$$A_{i_1\cdots i_k}(x):= 
X_1(x)\wedge\cdots\wedge\widehat{X_{i_1}}(x)\wedge\cdots\wedge\widehat{X_{i_k}}(x)\wedge\cdots\wedge X_m(x)\in
\Lambda^{m-k}A_x.$$
Then $\iota_{P(x)}^{\hspace*{.5mm}\text{t}}\mu(x)(A_{i_1\cdots i_k}(x))= (-1)^\alpha\alpha_{i_1\cdots i_k}\mu(x)$ for some $\alpha\in\mathbb{Z}$. Hence we conclude that $\flat_\mu(P)\neq 0$. In other words,  $\flat_\mu$ is injective.
\end{proof}

Notice that  $\flat_\mu$ is $C^\infty(M)$-linear and hence yields a vector bundle isomorphism  
$$\flat_\mu \colon \Lambda^k A \longrightarrow \Lambda^{m-k} A^\ast\otimes O(A).$$

Now we use the above isomorphisms and proceed as in Section $4$, to define an operator
\[
 \delta_\mu:= \flat_\mu^{-1}\circ d^{\hspace*{.5mm}\text{t}}_{A}\circ \flat_\mu\colon  \Gamma (\Lambda^k A)\longrightarrow \Gamma (\Lambda^{k-1} A)
\]
to deduce a homology complex $(\Gamma^\bullet A, \delta_\mu)$ from the twisted Lie algebroid cochain complex $(\Gamma(\Lambda^\bullet A^\ast \otimes O(A)), d_A^{\hspace*{.5mm}\text{t}}).$ Define  
\[
 \mathcal{V}_{\hspace*{.5mm}\text{t}}^k(A,\Pi)= \{P\in\Gamma (\Lambda^k A)|~\iota^{\hspace*{.5mm}\text{t}}_\alpha P=0\hspace*{1mm}\text{for all}\hspace*{1mm}\alpha\in
\Gamma (A^\ast),\alpha\in\hspace*{1mm}\text{ker}~\Pi_1\},~k\geq 1
\]
\noindent Set $\mathcal{V}_{\hspace*{.5mm}\text{t}}^0(A,\Pi)= C^\infty(M)$. Then one can verify that 
$$\delta_\mu(\mathcal{V}_{\hspace*{.5mm}\text{t}}^k(A,\Pi)) \subseteq \mathcal{V}_{\hspace*{.5mm}\text{t}}^{k-1}(A,\Pi)~~\mbox{for all}~~k,$$
yielding a subcomplex $(\mathcal{V}_{\hspace*{.5mm}\text{t}}^\bullet(A, \Pi), \delta_\mu)$ of $(\Gamma^\bullet A, \delta_\mu)$.
The canonical Nambu homology modules in the present context are defined by the above subcomplex.

We may now recover all the results in Section $4$ and their proofs can be written verbatim with the tools as developed above and the duality theorem in the present context may be deduced in the following form. 

\begin{thm}
Let $(A,[,]_{A},a)$ be a Lie algebroid over a smooth manifold $M$ with a regular Nambu structure $\Pi$ of order $n,~3\leq n\leq m = rank~A$. Assume that $\Gamma A^\ast$ is generated by elements of the form $d_Af$ where $f\in C^\infty(M)$. If the modular class of $A$ is null then
\[
 \mathcal{H}_{\text{N}}^k(A) \cong \mathcal{H}_{\Pi}^k(A) \cong \mathcal{H}_{n-k}^\text{canN}(A)
\hspace*{2mm}\text{for all}\hspace*{2mm}k= 0,\ldots,n.
\]
\end{thm}
%----------------------------------------

{\bf Acknowledgement:} Second author would like to thank Mahuya Datta for clearing some initial doubts while learning the subject.

\newpage 
\mbox{ }\\

\providecommand{\bysame}{\leavevmode\hbox to3em{\hrulefill}\thinspace}
\providecommand{\MR}{\relax\ifhmode\unskip\space\fi MR }
\providecommand{\MRhref}[2]{%
  \href{http://www.ams.org/mathscinet-getitem?mr=#1}{#2}
}
\providecommand{\href}[2]{#2}

\end{document}